\documentclass[a4paper,reqno,11pt]{amsart}

\usepackage{amsmath, amsfonts, amssymb, amsthm, amscd}
\usepackage{graphicx}
\usepackage{psfrag}
\usepackage{perpage}
\usepackage{url}
\usepackage{color}

\usepackage[a4paper,scale={0.72,0.74},marginratio={1:1},footskip=7mm,headsep=10mm]{geometry}

\usepackage{dsfont} 

\usepackage{hyperref}

\numberwithin{equation}{section}

\newtheorem{theorem}{Theorem}[section]
\newtheorem{lemma}[theorem]{Lemma}
\newtheorem{proposition}[theorem]{Proposition}

\newtheorem{remark}[theorem]{Remark}

\newcommand{\cI}{{\ensuremath{\mathcal I}} }

\newcommand{\cR}{{\ensuremath{\mathcal R}} }

\renewcommand{\tilde}{\widetilde}          
\DeclareMathSymbol{\leqslant}{\mathalpha}{AMSa}{"36} 
\DeclareMathSymbol{\geqslant}{\mathalpha}{AMSa}{"3E} 
\DeclareMathSymbol{\eset}{\mathalpha}{AMSb}{"3F}     

\newcommand{\R}{\mathbb{R}}

\newcommand{\Z}{\mathbb{Z}}
\newcommand{\N}{\mathbb{N}}
\newcommand{\Q}{\mathbb{Q}}

\newcommand{\PEfont}{\mathrm}

\renewcommand\P{\ensuremath{\PEfont P}}

\newcommand{\ind}{\mathds{1}}

\renewcommand{\epsilon}{\varepsilon} 
\renewcommand{\theta}{\vartheta} 
\renewcommand{\rho}{\varrho} 
\renewcommand{\phi}{\varphi}

\newcommand{\rr}{\mathsf{r}}
\newcommand{\RR}{\mathsf{R}}
\newcommand{\sfC}{\mathsf{C}}

\newenvironment{myenumerate}{%
\renewcommand{\theenumi}{\arabic{enumi}}%
\renewcommand{\labelenumi}{{\rm(\theenumi)}}%
\begin{list}{\labelenumi}
	{%
	\setlength{\itemsep}{0.4em}%
	\setlength{\topsep}{0.5em}%
	\setlength\leftmargin{2.45em}%
	\setlength\labelwidth{2.05em}%
	\setlength{\labelsep}{0.4em}%
	\usecounter{enumi}%
	}%
	}%
{\end{list}
}

{\end{myenumerate}}

\newenvironment{myitemize}{%
\begin{list}{$\bullet$}%
 	{%
	\setlength{\itemsep}{0.4em}%
	\setlength{\topsep}{0.5em}%
	\setlength\leftmargin{2.45em}%
	\setlength\labelwidth{2.05em}%
	\setlength{\labelsep}{0.4em}%
	}%
	}%
{\end{list}}

\renewenvironment{itemize}{
\begin{myitemize}}%
{\end{myitemize}}

\MakePerPage[2]{footnote} 

\title{The strong renewal theorem}

\author{Francesco Caravenna}

\address{Dipartimento di Matematica e Applicazioni, Universit\`a
degli Studi di Milano-Bicocca, via Cozzi 55, 20125 Milano, Italy}

\email{francesco.caravenna@unimib.it}



\date{\today}

\newcommand\dd{\mathrm{d}}

\begin{document}

\begin{abstract}
We consider real random walks with positive increments (renewal processes)
in the domain of attraction of a stable law with index $\alpha \in (0,1)$.
The famous local renewal theorem of Garsia and Lamperti \cite{cf:GL},
also called \emph{strong renewal theorem},
is known to hold in complete generality only for $\alpha > \frac{1}{2}$.
Understanding when the strong renewal theorem holds
for $\alpha \le \frac{1}{2}$
is a long-standing problem, with sufficient conditions given
by Williamson \cite{cf:Wil}, Doney \cite{cf:D} and Chi \cite{cf:Chi0,cf:Chi}.
In this paper we give a complete solution, providing explicit
necessary and sufficient conditions
(an analogous result has been independently and simultaneously proved by
Doney~\cite{cf:Don}). We also show
that these conditions fail to be sufficient if the random walk is allowed to take
negative values.

\emph{This paper is superseded by \cite{cf:CD}.}
\end{abstract}

\maketitle

\section{Introduction}

We use the notations $\N = \{1,2,3,\ldots\}$
and $\N_0 = \N \cup \{0\}$. 
Given two functions $f,g:[0,\infty) \to (0,\infty)$ we write
$f \sim g$ to mean $\lim_{s\to\infty} f(s)/g(s) = 1$.

We denote by $\cR_\gamma$ the space of
regularly varying functions with index $\gamma \in \R$, that is
$f \in \cR_\gamma$ if and only if
$\lim_{x\to\infty} f(\lambda x)/f(x) = \lambda^\gamma$ for all $\lambda \in (0,\infty)$.
Functions in $\cR_0$ are called slowly varying.
Note that $f \in \cR_\gamma$ if and only if
$f(x) = x^\gamma \ell(x)$ for some slowly varying function $\ell \in \cR_0$.
We refer to \cite{cf:BinGolTeu} for more details.

\subsection{Main result}
\label{sec:setting}

We fix a probability $F$ on $[0,\infty)$
and we let $X, (X_i)_{i\in\N}$ be independent and identically
distributed (i.i.d.) random variables with law $F$.
The associated random walk (renewal process)
will be denoted by $S_n := X_1 + \ldots + X_n$, with $S_0 := 0$.
We say that $F$ is \emph{arithmetic} 
if it is supported by $h\Z$ for some $h > 0$,
and the maximal value of $h>0$ with this property is called the 
\emph{arithmetic span}  of $F$.

Our key assumption is that
there exist $\alpha \in (0,1)$ and $A \in \cR_\alpha$ such that
\begin{gather}
	\label{eq:tail1}
	\overline F(x) :=
	F((x,\infty)) = 
	\P(X > x) \sim \frac{1}{A(x)} 
	\qquad \text{as} \quad x\to \infty \,.
\end{gather}
We can write $A(x) = L(x) \, x^\alpha$,
for a suitable slowly varying function $L \in \cR_0$.
By \cite[\S1.3.2]{cf:BinGolTeu},
we may take $A:[0,\infty) \to (0,\infty)$ to be differentiable,
strictly increasing and
\begin{equation}\label{eq:deriv}
	A'(s) \sim \alpha \frac{A(s)}{s} \,, \qquad \text{as } s \to \infty \,.
\end{equation}

Let us introduce the renewal measure
\begin{equation} \label{eq:U}
	U(\dd x) := \sum_{n\ge 0} \P(S_n \in \dd x) \,,
\end{equation}
so that $U(I)$ is the expected number of variables $S_n$ that fall
inside $I \subseteq \R$.
It is well known \cite[Theorem 8.6.3]{cf:BinGolTeu} that \eqref{eq:tail1} implies
the following infinite-mean version of the
renewal theorem, with $\sfC = \sfC(\alpha) = \frac{\alpha\sin(\pi\alpha)}{\pi}$:
\begin{equation} \label{eq:SRTint}
	U([0,x]) \sim \frac{\sfC}{\alpha} \, A(x) \quad \text{as} \quad x\to\infty \,.
\end{equation}

Let us introduce the shorthand
\begin{equation}\label{eq:I}
	I = (-h,0] \qquad \text{where} \qquad h :=
	\begin{cases}
	\text{arithmetic span of $F$} & \text{(if $F$ is arithmetic)}\\
	\text{any fixed number $> 0$} &
	\text{(if $F$ is non-arithmetic)} \,.
	\end{cases}
\end{equation}
Recalling \eqref{eq:deriv}, it is natural to look for
a local version of \eqref{eq:SRTint}, namely
\begin{equation} \label{eq:SRT}
	U(x+I) = U((x-h,x]) \sim \sfC \, h \, \frac{A(x)}{x} \quad \text{as} \quad x \to \infty \,.
\end{equation} 
This relation, called \emph{strong renewal theorem (SRT)},
is known to hold in complete generality
under \eqref{eq:tail1} when $\alpha > \frac{1}{2}$, cf.\ \cite{cf:GL,cf:Wil,cf:Eri}.
On the other hand, when $\alpha \le \frac{1}{2}$ there are examples
of $F$ satisfying \eqref{eq:tail1} but not \eqref{eq:SRT}.
It is therefore of great
theoretical and practical interest to find conditions on $F$,
in addition to \eqref{eq:tail1},
ensuring the validity of \eqref{eq:SRT} for $\alpha \le \frac{1}{2}$.

Let us introduce the function
\begin{equation}\label{eq:rr}
	\rr(x) := \frac{F(x+I)}{\overline{F}(x)/x}
	\sim x \, A(x) \, F(x+I) \,.
\end{equation}
By \eqref{eq:tail1}, one expects $\rr(x)$ to be bounded for 
``typical'' values of $x$, although there might be exceptional values for which it
is much larger.
It is by now a classical result that a sufficient condition for the SRT \eqref{eq:SRT}
is the global boundedness of $\rr$:
\begin{equation}\label{eq:doney}
	\sup_{x \ge 0} \, \rr(x) < \infty \,,
\end{equation}
as proved by Doney \cite{cf:D} in the arithmetic case
(extending Williamson \cite{cf:Wil}, who assumed $\alpha > \frac{1}{4}$)
and by Vatutin and Topchii \cite{cf:VT} in the
non-arithmetic case.

More recently \cite{cf:Chi0,cf:Chi}, Chi 
showed that \eqref{eq:doney} can be substantially relaxed, 
through suitable integral criteria.
To mention the simplest \cite[Theorem 1.1]{cf:Chi},
if one defines
\begin{equation}\label{eq:RR}
	\RR_T(a,b) := \int_a^b \big(\rr(y)-T\big)^+
	\, \dd y \,,
	\qquad \text{where} \quad
	z^+ := \max\{z,0\} \,,
\end{equation}
a sufficient condition for the 	SRT \eqref{eq:SRT}, for $\alpha \le \frac{1}{2}$,
is that for some
$\eta \in (0,1)$, $T \in [0,\infty)$
\begin{equation}\label{eq:chi}
\begin{split}
	\RR_T((1-\eta) x, x)
	\underset{x\to\infty}{=}
	\begin{cases}
	o(A(x)^2) & \text{ if } \alpha < \frac{1}{2} \\
	\rule{0pt}{1.6em}
	o\Big(\frac{A(x)^2}{u(x)}\Big) 
	& \text{ if } \alpha = \frac{1}{2}
	\end{cases}
	\qquad \rule{0pt}{1.4em}\text{where} \quad
	\textstyle u(x) := \int_1^x \frac{A(s)^2}{s^2} \, \dd s \,.
\end{split}
\end{equation}
This clearly improves \eqref{eq:doney} (just note that $\RR_T(a,b) \equiv 0$
for $T = \sup_{x\ge 0} \rr(x)$).
We refer to \cite{cf:Chi0,cf:Chi} for a variety of more general
(and more technical) sufficient conditions.

\smallskip

Integral criteria like \eqref{eq:chi} are appealing, because they are very explicit
and can be easily checked in concrete examples.
It is natural to ask whether more refined integral criteria
can provide \emph{necessary and sufficient} conditions
for the SRT \eqref{eq:SRT}. Our main result shows that this is indeed the case,
giving a complete solution to the SRT problem.

\begin{theorem}[Strong Renewal Theorem]\label{th:main}
Let $F$ be a probability on $[0,\infty)$
satisfying \eqref{eq:tail1} with $A \in \cR_\alpha$, for $\alpha \in (0, 1)$.
Define $I = (-h,0]$ with $h > 0$ as in \eqref{eq:I}.
\begin{itemize}
\item If $\alpha  > \frac{1}{2}$, the SRT \eqref{eq:SRT} 
holds with no extra assumption on $F$.

\item If $\alpha \le \frac{1}{2}$,
the SRT \eqref{eq:SRT}  holds if and only if 
\begin{equation}\label{eq:ns12'bis0}
	\lim_{\eta \to 0}
	\left\{ \limsup_{x\to\infty} \frac{1}{A(x)^2}
	\left( 
	\int_{1}^{\eta x} 
	\frac{A(s)^2}{s} \, \rr(x - s) \,
	\dd s 
	\right) 
	\right\} = 0 \,.
\end{equation}
\end{itemize}
For $\alpha < \frac{1}{2}$, setting $\RR_0(a,b) := \int_a^b \rr(y) \, \dd y$
(see \eqref{eq:RR}),
relation \eqref{eq:ns12'bis0} is equivalent to
\begin{equation}\label{eq:nsbis0}
	\lim_{\eta \to 0}
	\left\{ \limsup_{x\to\infty} \frac{1}{A(x)^2}
	\left( \int_1^{\eta x} 
	\frac{A(s)^2}{s^2} \, \RR_0(x-s,x) \, \dd s
	\right) \right\} = 0 \,,
\end{equation}
while for $\alpha = \frac{1}{2}$ relation \eqref{eq:nsbis0} is stronger than
(i.e.\ it implies) \eqref{eq:ns12'bis0}.
\end{theorem}

\smallskip

In Section~\ref{sec:refo}
we reformulate conditions \eqref{eq:ns12'bis0}-\eqref{eq:nsbis0}
more explicitly in terms of the probability $F$ (see Lemma~\ref{lem:tec}).
We also present an overview on the strategy of the proof
of Theorem~\ref{th:main}, which is a refinement of the probabilistic approach
of Chi~\cite{cf:Chi0,cf:Chi} and 
allows to treat the arithmetic and non-arithmetic cases in a unified way,
avoiding characteristic functions (except for their implicit use in local limit theorems). 

In the rest of the introduction, after some remarks,
we derive some consequences 
of conditions \eqref{eq:ns12'bis0}-\eqref{eq:nsbis0}, see \S\ref{sec:soco}. 
Then we discuss the case of two-sided random walks,
showing that \emph{condition \eqref{eq:ns12'bis0} is not sufficient
for the SRT}, see \S\ref{sec:two-sided}.

\begin{remark}\rm
A result analogous to Theorem~\ref{th:main} has been independently and simultaneously proved by
Doney~\cite{cf:Don}.
\end{remark}

\begin{remark}\rm
When $\alpha > \frac{1}{2}$ condition \eqref{eq:ns12'bis0} follows
from \eqref{eq:tail1} (see
the Appendix~\S\ref{sec:equi<}).
As a consequence, we can reformulate Theorem~\ref{th:main} as follows:
\emph{assuming \eqref{eq:tail1}, condition \eqref{eq:ns12'bis0} is necessary and sufficient
for the SRT \eqref{eq:SRT} for any $\alpha \in (0,1)$}.\qed
\end{remark}

\begin{remark}\rm
The double limit $x \to \infty$ followed by $\eta \to 0$
can be reformulated as follows:
relations \eqref{eq:ns12'bis0}-\eqref{eq:nsbis0} are equivalent to
asking that, for any fixed function $g(x) = o(x)$,
\begin{align}\tag{\ref{eq:ns12'bis0}'}\label{eq:ns12'bis2}
	\int_{1}^{g(x)} 
	\frac{A(s)^2}{s} \, \rr(x - s) \,
	\dd s 
	& \underset{x\to\infty}{=} o\big(A(x)^2\big) \,,\\
	\tag{\ref{eq:nsbis0}'}\label{eq:nsbis2}
	\int_1^{g(x)} 
	\frac{A(s)^2}{s^2} \, \RR_0(x-s,x) \,
	\dd s 
	& \underset{x\to\infty}{=} o\big(A(x)^2\big) \,,
\end{align}
as an easy contradiction argument shows.\qed
\end{remark}

\begin{remark}\rm
Relations \eqref{eq:ns12'bis0}-\eqref{eq:nsbis0}
contain no cutoff parameter $T$, unlike \eqref{eq:chi}.
This can be introduced replacing
$\rr(x - s)$ by $( \rr(x - s) - T )^+$
and $\RR_0(x-s,x)$ by $\RR_T(x-s,x)$, respectively,
because \eqref{eq:ns12'bis0}-\eqref{eq:nsbis0}
are equivalent to the following:
\begin{align}\tag{\ref{eq:ns12'bis0}''}\label{eq:ns12'bis}
	\exists T \in [0,\infty): \qquad 
	& \lim_{\eta \to 0}
	\left\{ \limsup_{x\to\infty} \frac{1}{A(x)^2}
	\int_{1}^{\eta x} 
	\left( \frac{A(s)^2}{s} \, \big( \rr(x - s) - T \big)^+
	\right) \dd s \right\} = 0 \,,\\
	\tag{\ref{eq:nsbis0}''}\label{eq:nsbis}
	\exists T \in [0,\infty): \qquad 
	& \lim_{\eta \to 0}
	\left\{ \limsup_{x\to\infty} \frac{1}{A(x)^2}
	\int_1^{\eta x} 
	\left( \frac{A(s)^2}{s^2} \, \RR_T(x-s,x) \right) \dd s \right\} = 0 \,.
\end{align}
This is easily checked, by writing
\begin{equation*}
	\rr(x-s) \le T \,+\, \big(\rr(x-s)-T\big)^+ \,,
	\qquad
	\RR_0(x-s,x) \le Ts \,+\, \RR_T(x-s,x) \,,
\end{equation*}
and noting that the terms $T$ and $Ts$
give a negligible contribution to \eqref{eq:ns12'bis0} and \eqref{eq:nsbis0}, respectively,
because by Karamata's Theorem \cite[Proposition~1.5.8]{cf:BinGolTeu}
\begin{equation}\label{eq:basi}
	\int_1^{\eta x} \frac{A(s)^2}{s} \, \dd s \sim \frac{1}{2\alpha} A(\eta x)^2
	\sim \frac{1}{2\alpha} \eta^{2\alpha} A(x)^2 \qquad
	\text{as } x \to \infty\,.
\end{equation}

Nothing is really gained with the cutoff $T$, 
since relations  \eqref{eq:ns12'bis}-\eqref{eq:nsbis} 
are \emph{equivalent} to the $T=0$ versions \eqref{eq:ns12'bis0}-\eqref{eq:nsbis0}.
However,
in concrete examples it is often convenient to use
\eqref{eq:ns12'bis}-\eqref{eq:nsbis},
because they allow to focus one's attention
on the ``large'' values of $\rr$.\qed
\end{remark}

\subsection{Some consequences}
\label{sec:soco}

An immediate corollary of Theorem~\ref{th:main} is the sufficiency
of conditions \eqref{eq:doney} and \eqref{eq:chi} for the SRT \eqref{eq:SRT}.
\begin{itemize}
\item For condition \eqref{eq:doney}, note that it implies \eqref{eq:ns12'bis0},
thanks to \eqref{eq:basi}.

\item For condition \eqref{eq:chi}, note that it implies \eqref{eq:nsbis},
since
$\RR_T(x-s,x) \le \RR_T((1-\eta)x,x)$ and, moreover,
$\lim_{x\to\infty} u(x) = \int_1^{\infty} (A(s)^2/s^2) \, \dd s < \infty$
for $\alpha < \frac{1}{2}$,
because $A(s)^2/s^2$ is regularly varying with index $2\alpha-2 < -1$
(see \cite[Proposition 1.5.10]{cf:BinGolTeu}).
\end{itemize}

More generally, all the sufficient conditions presented in \cite{cf:Chi0,cf:Chi}
can be easily derived from Theorem~\ref{th:main}.
We present alternative
sufficient conditions, in terms of ``smoothness'' properties
of $F$. Observe that, if \eqref{eq:tail1} holds,
for any $s_x = o(x)$ one has
\begin{equation} \label{eq:deca}
	\frac{F((x,x+s_x])}{F((x,\infty))} =
	\frac{\P(X \in (x,x+s_x]) }{\P(X \in (x,\infty)) } \xrightarrow[x\to\infty]{} 0 \,.
\end{equation}
Our next result shows that a suitable polynomial rate of decay in \eqref{eq:deca} ensures the
validity of the SRT \eqref{eq:SRT}.
(Analogous conditions, in a different context, appear in \cite{cf:CSZ}).

\begin{proposition}\label{pr:main1}
Let $F$ be a probability on $[0,\infty)$ satisfying \eqref{eq:tail1}
for some $\alpha \in (0,\frac{1}{2}]$. A sufficient condition for
the SRT \eqref{eq:SRT} is that there is $\epsilon > 0$ such that,
for any $1 \le s_x = o(x)$,
\begin{equation} \label{eq:suff0}
	\frac{F((x,x+s_x])}{F((x,\infty))}
	= O \bigg(\left(\frac{s_x}{x}\right)^{1-2\alpha+\epsilon} \bigg) 
	\qquad \text{as} \quad x \to \infty \,.
\end{equation}
\end{proposition}

We finally focus on the case $\alpha = \frac{1}{2}$.
Our next result unravels this case,
by stating under which conditions on $A(x)$ the SRT \eqref{eq:SRT}
holds with no further assumption on $F$ than \eqref{eq:tail1}
(like it happens for $\alpha > \frac{1}{2}$).
Given a function $L$, let us define
\begin{equation}\label{eq:L*}
	L^*(x) := \sup_{1 \le s \le x} \, L(s) \,.
\end{equation}

\begin{theorem}[Case $\alpha = \frac{1}{2}$]\label{th:1/2}
Let $F$ be a probability on $[0,\infty)$ satisfying \eqref{eq:tail1}
with $\alpha = \frac{1}{2}$, that is $A \in \cR_{1/2}$.
Write $A(x) = L(x) \sqrt{x}$,
where $L \in \cR_0$ is slowly varying.
\begin{itemize}
\item If $A(x)$ satisfies the following condition:
\begin{equation}\label{eq:cond}
	L^*(x) \underset{x\to\infty}{=} O(L(x)) \,, 
\end{equation}
the SRT \eqref{eq:SRT} holds with no extra assumption on $F$.

\item If condition \eqref{eq:cond} fails, there are examples of $F$ 
for which the SRT \eqref{eq:SRT} fails. 
\end{itemize}
\end{theorem}

\begin{remark}\rm
Condition \eqref{eq:cond} is satisfied, in particular, when $A(x) \sim c \sqrt{x}$ for some
$c \in (0,\infty)$,
hence the SRT \eqref{eq:SRT} holds with no extra assumption on $F$, in this case.
\end{remark}

In order to understand how \eqref{eq:cond} arises,
we bound the integral in \eqref{eq:nsbis} from above by
$L^*(x)^2 R((1-\eta)x,x)$, hence a sufficient condition for the SRT \eqref{eq:SRT} is
\begin{equation}\label{eq:suffcond1/2}
	\exists T \in [0,\infty): \qquad
	\lim_{\eta \to 0} \left(
	\limsup_{x\to\infty} \,
	\frac{\displaystyle L^*(x)^2}{A(x)^2} \;
	\RR_T\big( (1-\eta)x,x \big) \right) = 0 \,,
\end{equation}
and a slightly weaker (but more explicit)
sufficient condition is
\begin{equation}\label{eq:suffcond1/2bis}
	\exists \eta \in (0,1), \ T \in [0,\infty) : \qquad
	\RR_T\big( (1-\eta)x,x \big) \underset{x\to\infty}{=}
	o \left( \frac{A(x)^2}{\displaystyle L^*(x)^2} \right) \,.
\end{equation}
It is worth observing that \eqref{eq:suffcond1/2}-\eqref{eq:suffcond1/2bis} 
refine Chi's condition \eqref{eq:chi} for $\alpha = \frac{1}{2}$,
because it is easy to show that $u(x) \ge c \, L^*(x)^2$ for some $c \in (0,\infty)$.

\subsection{Beyond renewal processes}
\label{sec:two-sided}

It is natural to consider the two-sided version of \eqref{eq:tail1},
i.e. to take a probability $F$ on the real line $\R$ which 
is in the domain of attraction of a stable law with index $\alpha \in (0,1)$
and positivity parameter $\rho \in (0,1]$.
More explicitly, setting $F(x) := F((-\infty,x])$ and $\overline{F}(x) :=
F((x,\infty))$, assume that
\begin{equation}\label{eq:tail2}
	\overline{F}(x) \sim \frac{p}{A(x)} \qquad \text{and} \qquad
	F(-x) \sim \frac{q}{A(x)} \qquad \text{as} \quad x \to \infty \,,
\end{equation}
where $A \in \cR_\alpha$ and $p > 0$, $q \ge 0$ are finite constants.
As usual, let $S_n = X_1 + \ldots + X_n$ be the random walk associated
to $F$ and define the renewal measure $U(\cdot)$ as in \eqref{eq:U}.

The ``integrated'' renewal theorem \eqref{eq:SRTint} still holds
(with a different value of $\sfC = \sfC(\alpha,\rho)$)
and, for $\alpha > \frac{1}{2}$,  the SRT \eqref{eq:SRTeq} follows again
by \eqref{eq:tail2} with no additional assumptions
cf.\ \cite{cf:GL,cf:Wil,cf:Eri,cf:Eri2}
(we give an independent proof in Section~\ref{sec:>12}).

For $\alpha \le \frac{1}{2}$, our next result gives a necessary condition for the SRT, which
is shown to be strictly stronger than \eqref{eq:ns12'bis0},
when $q > 0$.
This means that \emph{condition \eqref{eq:ns12'bis0} is not sufficient
for the SRT} in the two-sided case \eqref{eq:tail2}.

\begin{theorem}[Two-sided case]\label{th:cont}
Let $F$ be a probability on $\R$ satisfying \eqref{eq:tail2} for some
$A \in \cR_\alpha$, with $\alpha \in (0,1)$ and $p > 0$, $q \ge 0$.
Define $I = (-h,0]$ with $h > 0$ as in \eqref{eq:I}.
\begin{itemize}
\item If $\alpha  > \frac{1}{2}$, the SRT \eqref{eq:SRT} 
holds with no extra assumption on $F$.

\item If $\alpha \le \frac{1}{2}$, a necessary condition for the SRT \eqref{eq:SRT} is the following:
\begin{equation}\label{eq:ns12'bis0two}
	\lim_{\eta \to 0}
	\left\{ \limsup_{x\to\infty} \frac{1}{A(x)^2}
	\left( 
	\int_{1}^{\eta x} 
	\frac{A(s)^2}{s} \, \big( \rr(x - s) +
	\ind_{\{q > 0\}} \rr(x + s) \big) \,
	\dd s 
	\right) 
	\right\} = 0 \,.
\end{equation}
There are examples of $F$ satisfying \eqref{eq:ns12'bis0} but not \eqref{eq:ns12'bis0two},
for which the SRT fails.
\end{itemize}
\end{theorem}

It is not clear whether \eqref{eq:ns12'bis0two} is also sufficient
for the SRT, or whether additional conditions (possibly on the
left tail of $F$) need to be imposed.

\subsection{Structure of the paper}

The paper is organized as follows.

\begin{itemize}
\item In Section~\ref{sec:prepa} we recall some standard background results.

\item In Section~\ref{sec:refo} we 
reformulate conditions \eqref{eq:ns12'bis0}-\eqref{eq:nsbis0}
and \eqref{eq:ns12'bis0two} more explicitly
in terms of $F$ (see Lemma~\ref{lem:tec})
and we describe the general strategy
underlying the proof of Theorem~\ref{th:main},
which is carried out in the following Sections~\ref{sec:>12}, \ref{sec:first} and~\ref{sec:gen}.

\item In Section~\ref{sec:prmain} we prove Proposition~\ref{pr:main1}
and Theorems~\ref{th:1/2} and~\ref{th:cont},
while the Appendix~\ref{sec:refoapp} contains the proofs of some 
auxiliary results.
\end{itemize}

\section{Setup}
\label{sec:prepa}

\subsection{Notation}

We write $f(s) \lesssim g(s)$ or $f \lesssim g$ to mean $f(s) = O(g(s))$, i.e.\
for a suitable constant $C < \infty$ one has
$f(s) \le C \, g(s)$ for all $s$ in the range under consideration.
The constant $C$ may depend on the probability $F$ (in particular, on $\alpha$)
and on $h$.
When some extra
parameter $\epsilon$ enters the constant $C = C_\epsilon$,
we write $f(s) \lesssim_\epsilon g(s)$.
If both $f \lesssim g$ and $g \lesssim f$, we write $f \simeq g$.
We recall that $f(s) \sim g(s)$ means $\lim_{s\to\infty} f(s)/g(s) = 1$.

\subsection{Regular variation}
\label{eq:regvar}
We recall that $A : [0,\infty) \to (0,\infty)$ in \eqref{eq:tail1} is
assumed to be differentiable, strictly increasing and such that \eqref{eq:deriv} holds.
For definiteness, let us fix $A(0) := \frac{1}{2}$ and $A(1) := 1$,
so that both $A$ and $A^{-1}$ map $[1,\infty)$ onto itself.

We observe that, by Potter's bounds, for
every $\epsilon > 0$ one has
\begin{equation} \label{eq:Potter}
	\rho^{\alpha+\epsilon} \lesssim_\epsilon \frac{A(\rho x)}{A(x)}
	\lesssim_\epsilon
	\rho^{\alpha-\epsilon} \,, \qquad
	\forall \rho \in (0,1], \ x \in (0,\infty)
	\ \text{ such that } \ \rho x \ge 1 \,.
\end{equation}
More precisely, part (i) of \cite[Theorem 1.5.6]{cf:BinGolTeu} shows that
relation \eqref{eq:Potter} holds for $\rho x \ge \bar x_\epsilon$, for
a suitable $\bar x_\epsilon < \infty$; the extension to
$1 \le \rho x \le \bar x_\epsilon$ follows as in part (ii) of the same theorem,
because $A(y)$ is bounded away from zero and infinity for $y \in [1, \bar x_\epsilon]$.

We also recall Karamata's Theorem
\cite[Proposition~1.5.8]{cf:BinGolTeu}: 
\begin{equation}\label{eq:Kar}
	\text{if } f(n) \in \cR_\zeta \text{ with } \zeta > -1: \qquad
	\sum_{n \le t} f(n) 
	\underset{t \to \infty}{\sim} \frac{1}{\zeta+1} \, t \, f(t) \,.
\end{equation}
As a matter of fact, this relation holds also in the limiting case $\zeta = -1$, in the sense that
$t f(t) = o(\sum_{n=1}^t f(n))$, by \cite[Proposition~1.5.9a]{cf:BinGolTeu}.

\subsection{Local limit theorems}
\label{sec:llt}

We call a probability $F$ on $\R$
\emph{lattice} if it is supported by $v\Z+a$ for some $v > 0$ and $0 \le a < v$,
and the maximal value of $v>0$ with this property is called the 
\emph{lattice span} of $F$.
If $F$ is arithmetic
(i.e.\ supported by $h\Z$, cf.\ \S\ref{sec:setting}),
then it is also lattice, but the spans might differ
(for instance, $F(\{-1\}) = F(\{+1\}) = \frac{1}{2}$ has 
arithmetic span $h=1$ and lattice span $v=2$).
A lattice distribution is not necessarily
arithmetic.\footnote{If $F$ is lattice, say
supported by $v\Z+a$
where $v$ is the lattice span and $a \in [0,v)$, then $F$ is arithmetic if and only if 
$a/v \in \Q$, in which case its arithmetic span equals
$h = v / m$ for some $m\in\N$.}

Let us define
\begin{equation*}
	a_n := A^{-1}(n) \,, \qquad n \in\N_0 \,,
\end{equation*}
so that $a_n \in \cR_{1/\alpha}$. Under \eqref{eq:tail1}
or, more generally, \eqref{eq:tail2}, $S_n/a_n$ converges
in distribution as $n\to\infty$ toward a stable
law, whose density we denote by $\phi$. If we set
\begin{equation}\label{eq:J}
	J = (-v,0] \qquad \text{with} \qquad v =
	\begin{cases}
	\text{lattice span of $F$} & \text{(if $F$ is lattice)}\\
	\text{any fixed number $> 0$} &
	\text{(if $F$ is non-lattice)} 
	\end{cases}\,,
\end{equation}
by Gnedenko's
and Stone's local limit theorems \cite[Theorems~8.4.1 and~8.4.2]{cf:BinGolTeu}
we have
\begin{equation}\label{eq:llt}
	\lim_{n\to\infty} \, \sup_{x \in \R}
	\left| a_n \, \P(S_n \in x+J) - v \, \phi \left( \frac{x}{a_n} \right)
	\right| = 0 \,.
\end{equation}
Since $\sup_{z\in\R} \phi(z) < \infty$, we obtain the useful estimate
\begin{equation}\label{eq:sup}
	\sup_{z\in\R} \P(S_n \in (x-w, x]) \lesssim_w \frac{1}{a_n} \,,
\end{equation}
which, plainly, holds for \emph{any} fixed $w > 0$ (not necessarily the lattice span of $F$).

Besides the local limit theorem \eqref{eq:llt}, a key tool in the proof will be
a local large deviations estimate
by Denisov, Dieker and Shneer \cite[Proposition 7.1]{cf:DDS}
(see \eqref{eq:DDS} below).

\smallskip
\section{Proof of Theorem~\ref{th:main}: strategy}
\label{sec:refo}

We start reformulating the key conditions \eqref{eq:ns12'bis0}-\eqref{eq:nsbis0}
and \eqref{eq:ns12'bis0two}
more explicitly in terms of $F$.
We recall that $X$ denotes a random variable with law $F$.
The next Lemma is proved
in Appendix~\ref{sec:tec}.

\begin{lemma}\label{lem:tec}
Assuming \eqref{eq:tail1}, condition \eqref{eq:ns12'bis0} is equivalent to
\begin{equation}\label{eq:ns12}
	\lim_{\eta \to 0}
	\left( \limsup_{x\to\infty} \frac{x}{A(x)}
	\int_{s \in [1,\eta x)} 
	\frac{A(s)^2}{s} \, \P(X \in x - \dd s) \right) = 0 \,,
\end{equation}
where we set $\P(X \in x - \dd s) := \P(x - X \in \dd s)$,
and condition \eqref{eq:nsbis0} is equivalent to
\begin{equation}\label{eq:ns}
	\lim_{\eta \to 0}
	\left( \limsup_{x\to\infty} \frac{x}{A(x)}
	\int_1^{\eta x} 
	\frac{A(s)^2}{s^2} \, \P(X \in (x-s,x]) \, \dd s \right) = 0 \,.
\end{equation}

Analogously, assuming \eqref{eq:tail2},
condition \eqref{eq:ns12'bis0two} is equivalent to
\begin{equation}\label{eq:ns12two}
	\lim_{\eta \to 0}
	\left( \limsup_{x\to\infty} \frac{x}{A(x)}
	\int_{s \in [1,\eta x)} 
	\frac{A(s)^2}{s} \, \big( \P(X \in x - \dd s)
	+ \ind_{\{q>0\}}\P(X \in x + \dd s) \big) \right) = 0 \,,
\end{equation}
\end{lemma}

It is now easy to prove the second part of
Theorem~\ref{th:main}: through a standard integration by parts, one
shows that if $\alpha < \frac{1}{2}$
relation \eqref{eq:ns} is equivalent to \eqref{eq:ns12}, while
if $\alpha = \frac{1}{2}$ it is stronger than \eqref{eq:ns12}. We refer
to the Appendix~\S\ref{sec:maramao} for the details.

\smallskip

Next we turn to the first part of Theorem~\ref{th:main},
i.e.\ the fact that \eqref{eq:ns12'bis0} is a necessary and sufficient
condition for the SRT.
The following general statement is known
\cite[Appendix]{cf:Chi0}: for $F$ satisfying \eqref{eq:tail1},
or more generally \eqref{eq:tail2}, the SRT \eqref{eq:SRT} is equivalent to
\begin{equation} \label{eq:SRTeq}
	\lim_{\delta \to 0}
	\Bigg( \limsup_{x\to\infty} \frac{x}{A(x)}
	\sum_{1 \le n \le A(\delta x)} \P(S_n \in x+I) \Bigg) = 0 \,,
\end{equation}
which means that small values of $n$
give a negligible contribution to the renewal measure
(we refer to Remark~\ref{rem:expl} below for an intuitive explanation of \eqref{eq:SRTeq}).
By Lemma~\ref{lem:tec}, it remains to show that
\emph{condition \eqref{eq:ns12} is necessary and sufficient
for \eqref{eq:SRTeq}}.

\smallskip

The necessity of \eqref{eq:ns12}
(or, if we assume \eqref{eq:tail2}, of \eqref{eq:ns12two})
is quite easy to check and is carried out in the Appendix~\ref{sec:necce}.
Showing the sufficiency of \eqref{eq:ns12} for \eqref{eq:SRTeq}
is much harder and is the core of the paper. 

\begin{itemize}
\item In Section~\ref{sec:>12} we prove that \eqref{eq:SRTeq} follows by
\eqref{eq:tail1} alone, if $\alpha > \frac{1}{2}$. The proof is based on the notion
of ``big jump'' and on two key bounds, cf.\ Lemmas~\ref{th:kale}
and~\ref{th:nobig}, that will be exploited
in an essential way also for the case $\alpha \le \frac{1}{2}$.

\item In Section~\ref{sec:first} we prove that \eqref{eq:ns12} implies \eqref{eq:SRTeq} 
in the special regime $\alpha \in (\frac{1}{3}, \frac{1}{2}]$. This case is technically
simpler, because there is only one big jump to deal with,
but it already contains all the ingredients of the general case.

\item In Section~\ref{sec:gen} we complete the proof,
showing that \eqref{eq:ns12} implies \eqref{eq:SRTeq} 
for any $\alpha \in (0, \frac{1}{2}]$. The strategy is conceptually
analogous to the one of Section~\ref{sec:first} but it is technically
much more involved, because we have to deal with more than one big jump.

\end{itemize}

\begin{remark}\rm\label{rem:usee}
Condition \eqref{eq:ns12},
equivalently \eqref{eq:ns12'bis0}, implies that for any fixed $w > 0$
\begin{equation}\label{eq:nec}
	\P(X \in (x-w,x]) = o\left(\frac{A(x)}{x}\right) \qquad
	\text{as } x \to \infty \,,
\end{equation}
as we prove in the Appendix~\S\ref{sec:equinec}.
This is not surprising, since \eqref{eq:nec} is a necessary
condition for the SRT \eqref{eq:SRT}, because
$U(x+I) \ge \P(S_1 \in x+I) = \P(X \in x+I)$.
In Appendix~\S\ref{sec:equinec} we also prove the following easy consequence
of \eqref{eq:nec}: for all fixed $m\in\N$ and $w > 0$
\begin{equation}\label{eq:nec2}
	\P(S_m \in (x-w,x]) = o\left(\frac{A(x)}{x}\right) \qquad
	\text{as } x \to \infty \,.
\end{equation}
Relations \eqref{eq:nec}-\eqref{eq:nec2} will be useful in
the next sections.\qed
\end{remark}

\begin{remark}\rm\label{rem:expl}
It is worth explaining how \eqref{eq:SRTeq} arises.
For fixed $\delta > 0$, by \eqref{eq:U} we can write
\begin{equation} \label{eq:Uheur}
	U(x+I) \ge \sum_{A(\delta x) < n \le A(\frac{1}{\delta}x)} \P(S_n \in x+I) \,.
\end{equation}
Since $\P(S_n \in x+I) \sim \frac{h}{a_n} \, \phi(\frac{x}{a_n})$
by \eqref{eq:llt}
(where we take $h=v$ for simplicity),
a Riemann sum approximation yields 
(see \cite[Lemma~3.4]{cf:Chi0} for the details)
\begin{equation*}
	\sum_{A(\delta x) < n \le A(\frac{1}{\delta}x)} \P(S_n \in x+I) \sim
	h \, \frac{A(x)}{x} \, \sfC(\delta) \,, \qquad \text{with} \qquad
	\sfC(\delta) = \alpha \int_\delta^{\frac{1}{\delta}}
	z^{\alpha-2} \phi(\tfrac{1}{z}) \, \dd z \,.
\end{equation*}
One can show that $\lim_{\delta \to 0} \sfC(\delta) =
\sfC$,
therefore proving the SRT \eqref{eq:SRT}
amounts to controlling the ranges excluded from \eqref{eq:Uheur},
i.e. $\{n \le A(\delta x)\}$ and $\{n > A(\frac{1}{\delta}x)\}$. The latter always gives
a negligible contribution, by
the bound $\P(S_n \in x+I) \le C/a_n$ (recall \eqref{eq:sup}),
and the former is controlled precisely by \eqref{eq:SRTeq}.\qed
\end{remark}

\section{Proof of Theorems~\ref{th:main}
and~\ref{th:cont}: the case $\alpha  > \frac{1}{2}$}
\label{sec:>12}

\emph{In this section we prove that, if $\alpha > \frac{1}{2}$,
relation \eqref{eq:SRTeq}, which is equivalent to
the SRT \eqref{eq:SRT}, follows with no additional assumptions by \eqref{eq:tail1}, or more generally
by \eqref{eq:tail2}} (we never use the positivity of the increments
of the random walk in this section).

We have to estimate
the probability of the event $\{S_n \in x+I\}$ with $n \le A(\delta x)$,
where $S_n = X_1 + X_2 + \ldots + X_n$.
Let us call ``big jump'' any increment $X_i$ strictly larger than
a suitable threshold $\xi_{n,x}$, defined as a multiplicative
average of $a_n$ and $x$:
\begin{equation} \label{eq:xi}
	\xi_{n,x} := a_n^{\gamma_\alpha} \, x^{1-\gamma_\alpha} =
	a_n \, \left( \frac{x}{a_n} \right)^{1-\gamma_\alpha} \,,
	\qquad \text{with} \quad
	\gamma_\alpha := \frac{\alpha}{4} \,
	\left( 1-\left\{ \frac{1}{\alpha} \right\} \right) > 0 \,,
\end{equation}
where $\{z\} := z - \lfloor z \rfloor \in [0,1)$ denotes the fractional part of $z$.
The reason for the specific choice of $\gamma_\alpha > 0$ will be clear later
(it is important that $\gamma_\alpha$ is small enough).

\subsection{Bounding the number of big jumps}

As a first step, for every $\alpha \in (0,1)$, we show that, on the event
$\{S_n \in x+I\}$ with $n \le A(\delta x)$, the number of ``big jumps''
can be bounded by a deterministic number $\kappa_\alpha \in \N_0$,
defined as follows:
\begin{equation} \label{eq:kappa}
	\kappa_\alpha := 
	\left\lfloor \frac{1}{\alpha} \right\rfloor - 1 \,.
	\qquad \text{i.e.} \qquad \kappa_\alpha = m
	\quad \text{if} \quad \alpha \in (\tfrac{1}{m+2}, \tfrac{1}{m+1}]
	\quad \text{with} \ \ m \in \N_0 \,.
\end{equation}
Note that $\kappa_\alpha = 0$ if $\alpha  > \frac{1}{2}$ and this
is why the SRT holds with no additional assumption in this case.
If $\alpha \le \frac{1}{2}$, on the other hand,
$\kappa_\alpha \ge 1$ and a more refined analysis is required.

Let us call $B_{n,x}^k$ the event
``there are exactly $k$ big jumps'', i.e.\
\begin{equation} \label{eq:B}
\begin{split}
	B_{n,x}^0 & := \left\{\max_{1 \le i \le n} X_i \le \xi_{n,x}
	\right\} \,, \\
	B_{n,x}^k & := \left\{\exists I \subseteq \{1,\ldots, n\}, \ |I| = k: \
	\min_{i\in I} X_i > \xi_{n,x} \,, \ \max_{j \in \{1,\ldots,n\} \setminus I}
	X_i \le \xi_{n,x} \right\}, \ \  k \ge 1 \,,
\end{split}
\end{equation}
and correspondingly let $B_{n,x}^{\ge k}$ be the event
``there are at least $k$ big jumps'':
\begin{equation} \label{eq:B>}
	B_{n,x}^{\ge k} := \bigcup_{\ell = k}^n B_{n,x}^\ell \,,
\end{equation}
The following lemma shows
that the event $B_{n,x}^{\ge \kappa_\alpha+1}$ gives
a negligible contribution to \eqref{eq:SRTeq} (just plug $\ell = 0$
and $m = \kappa_\alpha + 1$ into \eqref{eq:k}). This sharpens 
\cite[Lemma 4.1]{cf:Chi}, where $\kappa$ was defined as
$\lfloor \frac{1}{\alpha} \rfloor$, i.e.\ one unit \emph{larger} than our
choice \eqref{eq:kappa} of $\kappa_\alpha$.
Furthermore, we allow for an extra parameter $\ell$, that will be useful later.

\begin{lemma}\label{th:kale}
Let $F$ satisfy \eqref{eq:tail2} for some $A \in \cR_\alpha$, with $\alpha \in (0, 1)$
and $p > 0$, $q \ge 0$.
There is $\eta = \eta_\alpha > 0$
such that for all $\delta \in (0, 1]$, $x \in [1,\infty)$,
$\ell, m \in \N_0$ 
the following holds:
\begin{equation}\label{eq:k}
	\text{if \ $\ell + m \ge \kappa_\alpha + 1$}: \qquad
	\sum_{1 \le n \le A(\delta x)} n^\ell \,
	\P\left(S_n \in x+I, \, B_{n,x}^{\ge m} \right)
	\lesssim_{\ell,m} \delta^{\eta} \, \frac{A(x)^{\ell + 1}}{x} \,.
\end{equation}
\end{lemma}

\begin{proof}
Throughout the proof we work for $n \le A(\delta x)$, hence $a_n \le
\delta x \le x$ (since $\delta \le 1$). Consequently, recalling \eqref{eq:xi}, we have
$a_n \le \xi_{n,x} \le x$.

For $m \in \N$, recalling \eqref{eq:sup}, we can write
\begin{equation} \label{eq:plu}
\begin{split}
	\P\big(S_n \in x+I, \ & B_{n,x}^{\ge m} \big)
	= \P\left( S_n \in x+I, \, \exists A\subseteq \{1,\ldots, n\}, 
	|A|=m: \, \min_{i\in A} X_i > \xi_{n,x} \right) \\
	& \le n^m \, \P\left( S_n \in x+I, \
	\min_{1 \le i \le m} X_i > \xi_{n,x} \right) \\
	& = n^m \, \int_{w\in\R} \P\left(S_m \in \dd w, \, \min_{1 \le i \le m} X_i > \xi_{n,x} \right)
	\, \P(S_{n-m} \in x-w+I) \\
	& \le n^m \, \P\left(\min_{1 \le i \le m} X_i > \xi_{n,x} \right) \,
	\left\{ \sup_{z\in\R} \P(S_{n-m} \in z+I) \right\} \\
	& \lesssim n^m \, \P(X > \xi_{n,x})^m \frac{1}{a_{n-m}}
	\lesssim_m \frac{n^m}{A(\xi_{n,x})^m} \, \frac{1}{a_n} \,,
\end{split}
\end{equation}
and this estimate holds also for $m=0$
(in which case $\P(S_n \in x+I, \ B_{n,x}^{\ge 0} )
= \P(S_n \in x+I )$).
Next we apply the lower bound in \eqref{eq:Potter} with $\epsilon = \alpha$
and $\rho = \xi_{n,x}/x$
(note that the condition $\rho x = \xi_{n,x} \ge 1$ is fulfilled because
$\xi_{n,x} \ge a_n \ge 1$):
\begin{equation*}
	\frac{A(\xi_{n,x})}{A(x)} \gtrsim \left( \frac{\xi_{n,x}}{x} \right)^{2\alpha}
	= \left( \frac{a_n}{x} \right)^{2\alpha \gamma_\alpha} \,.
\end{equation*}
Looking back at \eqref{eq:plu}, we get 
\begin{equation}\label{eq:use}
\begin{split}
	\sum_{1 \le n \le A(\delta x)} n^\ell \, \P\left(S_n \in x+I, \ B_{n,x}^{\ge m} \right)
	& \lesssim \frac{x^{2\alpha \gamma_\alpha m}}{A(x)^{m}}
	\sum_{1 \le n \le A(\delta x)} \frac{n^{m+\ell}}{(a_n)^{2\alpha \gamma_\alpha m + 1}} \,.
\end{split}
\end{equation}

Since $a_n \in \cR_{1/\alpha}$, the sequence in the sum is
regularly varying with index
\begin{equation*}
	J'_{m,\ell,\alpha} =(m+\ell) - \frac{1}{\alpha}(2\alpha \gamma_\alpha m+1)
	= (m+\ell) (1-2\gamma_\alpha) - \frac{1}{\alpha} + 2 \gamma_\alpha \ell \,.
\end{equation*}
By assumption $\ell \ge 0$ and $m + \ell \ge \kappa_\alpha + 1$,
hence
\begin{equation}\label{eq:Jprime}
	J'_{m,\ell,\alpha} \ge J_\alpha \qquad \text{with} \qquad
	J_\alpha := (\kappa_\alpha + 1) (1-2\gamma_\alpha) - \frac{1}{\alpha} \,.
\end{equation}
We
claim that
\begin{equation}\label{eq:combin}
	J_\alpha > -1 \,, \qquad \text{that is} \qquad
	\kappa_\alpha + 1 > \frac{1-\alpha}{\alpha} \frac{1}{1-2\gamma_\alpha} \,.
\end{equation}
To verify it, write $\kappa_\alpha + 1 = \lfloor \frac{1}{\alpha}\rfloor =
\frac{1}{\alpha} - \{\frac{1}{\alpha}\} = \frac{1-\alpha}{\alpha} + (1-\{\frac{1}{\alpha}\})$,
so that relation
\eqref{eq:combin} becomes $1-\{\frac{1}{\alpha}\} >
\frac{1-\alpha}{\alpha}\frac{2\gamma_\alpha}{1-2\gamma_\alpha}$. Since $\gamma_\alpha < \frac{1}{4}$ by construction,
cf.\ \eqref{eq:xi}, we have
$\frac{1-\alpha}{\alpha}\frac{2\gamma_\alpha}{1-2\gamma_\alpha}
< \frac{1-\alpha}{\alpha} \, 4\gamma_\alpha
< \frac{4\gamma_\alpha}{\alpha}$ and it remains to note that
$\frac{4\gamma_\alpha}{\alpha} = 1-\{\frac{1}{\alpha}\}$,
by definition \eqref{eq:xi} of $\gamma_\alpha$.

Coming back to \eqref{eq:use}, since the sequence in the sum is regularly varying with index
$J'_{m,\ell,\alpha} \ge J_\alpha > -1$, we can apply relation \eqref{eq:Kar}, getting
\begin{equation}\label{eq:keycomp1}
\begin{split}
	\sum_{1 \le n \le A(\delta x)} n^\ell \, \P\left(S_n \in x+I, \ B_{n,x}^{\ge m} \right)
	& \lesssim_{\ell,m} \frac{x^{2\alpha \gamma_\alpha m}}
	{A(x)^{m}} \,
	\frac{A(\delta x)^{m+\ell+1}}
	{(\delta x)^{2\alpha \gamma_\alpha m + 1}} 
	= \frac{A(\delta x)^{\ell + m + 1}}{
	\delta^{2\alpha
	\gamma_\alpha m + 1} A(x)^m \, x} \,.
\end{split}
\end{equation}
It is convenient to introduce a parameter $b = b_\alpha \in [\frac{1}{2},1)$, 
depending only on $\alpha$, that will be fixed later.
Note that \eqref{eq:k} holds trivially for $\delta x < 1$
(the left hand side vanishes, due to $A(0) < 1$), hence we may assume that
$\delta x \ge 1$.
We can then apply the upper bound in \eqref{eq:Potter} with
$\epsilon = (1-b)\alpha$ and $\rho = \delta$, that is
$A(\delta x) \lesssim \delta^{b\alpha} A(x)$, which plugged into \eqref{eq:keycomp1} gives
\begin{equation}\label{eq:keycomp1b}
\begin{split}
	\sum_{1 \le n \le A(\delta x)} n^\ell \, \P\left(S_n \in x+I, \ B_{n,x}^{\ge m} \right)
	& \lesssim_{\ell,m} 	\delta^{b \alpha(\ell+m+1)-(2\alpha\gamma_\alpha m + 1)}
	\, \frac{A(x)^{\ell + 1}}{x} \\
	& = 	\delta^{\alpha\{(m+\ell) (b-2\gamma_\alpha) 
	- \frac{1}{\alpha} + 2 \gamma_\alpha \ell + 1\}}
	\, \frac{A(x)^{\ell + 1}}{x} \\
	& \le \delta^{\alpha\{(\kappa_\alpha + 1) (b-2\gamma_\alpha) - \frac{1}{\alpha} + 1\}}
	\, \frac{A(x)^{\ell + 1}}{x} \,,
\end{split}
\end{equation}
where the last inequality holds because $m + \ell \ge \kappa_\alpha + 1$
and $\ell \ge 0$ by assumption (recall that $\delta \le 1$
and note that $b-2\gamma_\alpha > 0$, because $b \ge \frac{1}{2}$ and $\gamma_\alpha < \frac{1}{4}$). 
Recalling \eqref{eq:Jprime}, we get
\begin{equation}\label{eq:keycomp1c}
	\sum_{1 \le n \le A(\delta x)} n^\ell \, \P\left(S_n \in x+I, \ B_{n,x}^{\ge m} \right)
	\lesssim_{\ell,m} 
	\delta^{\alpha\{(J_\alpha+1) - (1-b)(\kappa_\alpha + 1)\}}
	\, \frac{A(x)^{\ell + 1}}{x} \,.
\end{equation}
Since $J_\alpha + 1 > 0$, by \eqref{eq:combin}, we can choose $b<1$
so that the term in bracket
is strictly positive.
More explicitly, defining
$b = b_\alpha := \max\{\frac{1}{2}, 1- \frac{1}{2} \frac{J_\alpha+1}{\kappa_\alpha+1}\}$,
the right hand side of \eqref{eq:keycomp1c} becomes
$\lesssim \delta^{\alpha \{ \frac{1}{2}(J_\alpha+1) \}} \, \frac{A(x)^{\ell + 1}}{x}$.
This shows that
relation \eqref{eq:k} holds
with $\eta = \eta_\alpha := \frac{1}{2}\alpha(J_\alpha + 1)$.
\end{proof}

\subsection{The case of no big jumps}

Next we analyze the event 
$B_{n,x}^0$ of ``no big jumps'', showing that it gives
a negligible contribution to \eqref{eq:SRTeq}, irrespective of $\alpha \in (0,1)$.
(The extra parameter $\ell$ and the $\sup$ over $z$
in \eqref{eq:0} will be useful later.)

\begin{lemma}\label{th:nobig}
Let $F$ satisfy \eqref{eq:tail2} for some $A \in \cR_\alpha$, with $\alpha \in (0, 1)$
and $p > 0$, $q \ge 0$.
For all $\delta \in (0, 1]$, $x \in [1,\infty)$,
$\ell \in \N_0$ 
the following holds, with $\gamma = \gamma_\alpha > 0$ defined in \eqref{eq:xi}:
\begin{equation}\label{eq:0}
	\sum_{1 \le n \le A(\delta x)} 
	n^\ell \,
	\sup_{z  \ge \delta^{\gamma/2}x}
	\P\left(S_n \in z+I, \, B_{n,x}^0 \right)
	\lesssim e^{-\frac{1}{\delta^{\gamma/3}}} \, \frac{A(x)^{\ell+1}}{x} \,.
\end{equation}
\end{lemma}

\begin{proof}
Throughout the proof we may assume that $\delta x \ge 1$,
because for $\delta x < 1$ the left hand side of \eqref{eq:0}
vanishes (recall that $A(0) < 1$ by construction).

We need a refinement of \eqref{eq:llt}, given by
\cite[Proposition 7.1]{cf:DDS} (see also \cite[Lemma 3.2]{cf:Chi0}):
if $F$ satisfies \eqref{eq:tail1}, or more generally \eqref{eq:tail2},
there are $C_1, C_2 < \infty$ such that
for any sequence $s_n \to \infty$ and $z \ge 0$
\begin{equation}\label{eq:DDS}
	\P\left(S_n \in z+I, \, \max_{1 \le i \le n} X_i \le s_n \right)
	\le C_1 \, e^{C_2 \, \frac{n}{A(s_n)}} \left( \frac{1}{s_n} + \frac{1}{a_n} \right)
	e^{-\frac{z}{s_n}} \,.
\end{equation}
We choose $s_n = \xi_{n,x}$, cf.\ \eqref{eq:xi}.
For $n \le A(\delta x)$, with $\delta \le 1$,
we have $a_n \le x$, hence by \eqref{eq:xi} we get
$\xi_{n,x} \ge a_n$ and consequently $A(\xi_{n,x}) \ge n$. Applying \eqref{eq:DDS}, we obtain
\begin{equation} \label{eq:gia}
	\P\left(S_n \in z+I, \, B_{n,x}^0 \right)
	\le C_1 \, e^{C_2 \, \frac{n}{A(\xi_{n,x})}} \left( \frac{1}{\xi_{n,x}} + \frac{1}{a_n} \right)
	e^{-\frac{z}{\xi_{n,x}}} \lesssim
	\frac{1}{a_n} \, e^{-\frac{z}{x} \,(\frac{x}{a_n})^{\gamma}} 
	\,.
\end{equation}

The function $\phi(y) := \frac{1}{y} e^{-1/y^{\gamma}}$ is increasing
for $y \in (0, c]$, with $c \in (0,1)$ a fixed
constant (by direct computation $c = \gamma^{1/\gamma}$).
Then, if $\delta \le c^2$,
for $z \ge \delta^{\gamma/2} x$ and $a_n \le \delta x$ one has
\begin{equation*}
	\P\left(S_n \in z+I, \, B_{n,x}^0 \right)
	\le \frac{1}{a_n} \, e^{-(\frac{\sqrt{\delta} x}{a_n})^{\gamma}} 
	= \frac{1}{\sqrt{\delta}x} \phi \left(\frac{a_n}{\sqrt{\delta}x}\right)
	\le \frac{1}{\sqrt{\delta}x} \phi(\sqrt{\delta}) \,,
\end{equation*}
hence, always for $\delta \le c^2$, applying 
\eqref{eq:Potter} with $\epsilon = \alpha/2$ and $\rho = \delta$,
\begin{equation} \label{eq:ye}
\begin{split}
	\sum_{1 \le n \le A(\delta x)} \! n^\ell \!
	\sup_{z  \ge \delta^{\gamma/2}x} \P\left(S_n \in z+I, \, B_{n,x}^0 \right)
	& \lesssim
	\frac{\phi(\sqrt{\delta})}{\sqrt{\delta}x} \, A(\delta x)^{\ell+1}
	\lesssim \delta^{\ell \frac{\alpha}{2}} 
	\, \frac{e^{-\frac{1}{\delta^{\gamma/2}}}}{\delta^{1-\frac{\alpha}{2}}}
	\, \frac{A(x)^{\ell+1}}{x} \,.
\end{split}
\end{equation}
Bounding $\delta^{\ell\alpha} \le 1$ since $\delta \le 1$,
relation \eqref{eq:0} is proved for $\delta \le c^2$.

In case $\delta \in (c^2, 1]$, the right hand side of \eqref{eq:0} is
$\simeq A(x)^{\ell+1}/x$, hence we have to show that the left hand side
is $\lesssim A(x)^{\ell+1}/x$.
The contribution of the terms with $n \le A(c^2 x)$ is under control,
by \eqref{eq:ye} with $\delta = c^2$. For the remaining terms, by \eqref{eq:sup},
\begin{equation*}
	\sum_{A(c^2 x) < n \le A(\delta x)} n^\ell \,
	\sup_{z  \ge \delta^{\gamma/2}x} \P\left(S_n \in z+I, \, B_{n,x}^0 \right)
	\lesssim \sum_{A(c^2 x) < n \le A(\delta x)} \frac{n^\ell}{a_n} \le
	\frac{A(x)^{\ell+1}}{c^2 x} \,,
\end{equation*}
where we have bounded
$a_n \ge a_{A(c^2 x)} = c^2 x$
and $n \le A(\delta x) \le A(x)$ (recall that $\delta \le 1$). 
\end{proof}

\subsection{Proof of Theorems~\ref{th:main}
and~\ref{th:cont} for $\alpha > \frac{1}{2}$}
Assume \eqref{eq:tail1}, or more generally \eqref{eq:tail2}, for some $\alpha > \frac{1}{2}$.
We have already observed that $\kappa_\alpha = 0$ for $\alpha > \frac{1}{2}$, cf.\ \eqref{eq:kappa}.
We can then apply Lemma~\ref{th:kale} with $\ell = 0$ and $m=1$, 
since $\ell + m \ge \kappa_\alpha + 1$ in this case.
Together with Lemma~\ref{th:nobig} with $\ell = 0$ and $z=x$, this yields
\begin{equation} \label{eq:dec}
\begin{split}
	\sum_{1 \le n \le A(\delta x)} \P(S_n \in x+I) & =
	\sum_{1 \le n \le A(\delta x)} \P(S_n \in x+I, \, B_{n,x}^{\ge 1}) +
	\sum_{1 \le n \le A(\delta x)} \P(S_n \in x+I, \, B_{n,x}^{0}) \\
	& \lesssim \left( \delta^\eta + e^{-\frac{1}{\delta^{\gamma/3}}} \right)
	\frac{A(x)}{x} \,,
\end{split}
\end{equation}
which shows that relation \eqref{eq:SRTeq}, and hence
the SRT \eqref{eq:SRT}, holds true for $\alpha > \frac{1}{2}$.\qed

\section{Proof of Theorem~\ref{th:main}: 
sufficiency for $\alpha  \in (\frac{1}{3}, \frac{1}{2}]$}

\label{sec:first}

For $\alpha \le \frac{1}{2}$ big jumps have to be taken into account,
because $\kappa_\alpha \ge 1$, cf.\ \eqref{eq:kappa}, and we need to show
that their contributions can be controlled using \eqref{eq:ns12},
which is equivalent to \eqref{eq:ns12'bis0} by Lemma~\ref{lem:tec}.
In order to illustrate the main ideas, in this section
we focus on the special case $\alpha \in (\frac{1}{3}, \frac{1}{2}]$,
which is technically simpler, because $\kappa_\alpha = 1$.
The general case $\alpha \in (0,\frac{1}{2}]$ is treated in Section~\ref{sec:gen}.

\emph{Throughout this section we assume condition \eqref{eq:ns12}
and we show that, for $\alpha \in (\frac{1}{3},
\frac{1}{2}]$, it implies
\eqref{eq:SRTeq}, which is equivalent to the SRT \eqref{eq:SRT}.}

\smallskip

We start with a basic estimate.

\begin{lemma}\label{th:basic}
If $F$ satisfies \eqref{eq:tail1} with $\alpha \in (0,1)$, there are $C,c \in (0,\infty)$
such that for all $n\in\N_0$ and $z \in [0,\infty)$
\begin{equation}\label{eq:basic}
	\P(S_n \in z+I) \le \frac{C}{a_n} \, e^{-c \frac{n}{A(z)}} \,.
\end{equation}
\end{lemma}

\begin{proof}
Assuming that $n$ is even (the odd case is analogous)
and applying \eqref{eq:sup}, we get
\begin{equation*}
\begin{split}
	\P\left(S_{n} \in z+I \right) & = \int_{y \in [0,z]} \P(S_{\frac{n}{2}} \in \dd y ) \,
	\, \P(S_{\frac{n}{2}} \in z-y+I) \le \frac{1}{a_{\frac{n}{2}}} \,
	\P(S_{\frac{n}{2}} \le z) \\
	& \lesssim \frac{1}{a_n} \,
	\P\left(\max_{1 \le i \le \frac{n}{2}} X_i \le z \right)
	= \frac{\left(1 - \P(X>z) \right)^{\frac{n}{2}}}{a_n}
	\le \frac{e^{-\frac{n}{2} \P(X>z)}}{a_n}
	\le \frac{e^{-c \frac{n}{A(z)}}}{a_n} \,,
\end{split}
\end{equation*}
provided $c > 0$ is chosen such that $\P(X > z) \ge 2c/A(z)$ for all $z \ge 0$.
This is possible by \eqref{eq:tail1} and because $z \mapsto A(z)$ is
(increasing and) continuous,
with $A(0) > 0$ (see \S\ref{eq:regvar}).
\end{proof}

We are ready to prove that \eqref{eq:SRTeq} follows by \eqref{eq:ns12}
for $\alpha \in (\frac{1}{3}, \frac{1}{2}]$.
In analogy with \eqref{eq:dec}, we apply Lemma~\ref{th:kale}
with $\ell=0$ and, this time, with $m=2$, so that $\ell+m \ge \kappa_\alpha + 1$
(because $\kappa_\alpha = 1$). Applying also Lemma~\ref{th:nobig}
with $\ell = 0$ and $z=x$, we obtain
\begin{equation} \label{eq:dec2}
	\sum_{1 \le n \le A(\delta x)} \P(S_n \in x+I) 
	\lesssim \left( \delta^\eta + e^{-\frac{1}{\delta^{\gamma/3}}} \right)
	\frac{A(x)}{x} \,+\,
	\sum_{1 \le n \le A(\delta x)} \P(S_n \in x+I, \, B_{n,x}^1)  \,.
\end{equation}
The first term gives no problem
for \eqref{eq:SRTeq}, hence we focus on $\P(S_n \in x+I, \, B_{n,x}^1)$.
Plainly,
\begin{equation} \label{eq:cru}
\begin{split}
	\P(S_n \in x+I, \ B_{n,x}^1) & \le
	n \, \P\left(S_n \in x+I, \ X_n > \xi_{n,x}, \ 
	\max_{1 \le j \le n-1} X_i \le \xi_{n,x} \right) \\
	& = n \int_{y \in (\xi_{n,x},x]} \P(X \in \dd y) \,
	\P\left( S_{n-1} \in x-y+I, \ B_{n-1,x}^0\right) \,,
\end{split}
\end{equation}
where we recall that $B_{n-1,x}^0 =
\{\max_{1 \le j \le n-1} X_i \le \xi_{n,x}\}$.

We first consider the contribution to the integral 
given by $y \in (\xi_{n,x}, x(1-\delta^{\gamma/2})]$
(where $\gamma = \gamma_\alpha > 0$ was defined
in \eqref{eq:xi}): since $x-y \ge \delta^{\gamma/2} x$, this contribution is bounded by
\begin{equation} \label{eq:ehm}
	n \, \P(X > \xi_{n,x}) \,
	\sup_{z \ge \delta^{\gamma/2}x} \P\left( S_{n-1} \in z+I, \ B_{n-1,x}^0\right)
	\lesssim \sup_{z \ge \delta^{\gamma/2}x} \P\left( S_{n-1} \in z+I, \ B_{n-1,x}^0\right) \,,
\end{equation}
because $\P(X > \xi_{n,x}) \le \P(X > a_n) \sim 1/A(a_n) = 1/n$, since
$\xi_{n,x} \ge a_n (x/a_n)^{1-\gamma} \ge a_n$
for $n \le A(\delta x)$ with $\delta \le 1$.
Applying Lemma~\ref{th:nobig} with $\ell=0$, by \eqref{eq:cru} we get
\begin{equation} \label{eq:2lint}
\begin{split}
	& \sum_{1 \le n \le A(\delta x)} 
	\P(S_n \in  x+I, \ B_{n,x}^1) \lesssim
	e^{-\frac{1}{\delta^{\gamma/3}}} \, \frac{A(x)}{x} \,+\,
	\cI_{\delta,x} \\
	& \text{where} \qquad \cI_{\delta,x} :=
	\int_{y \in (x(1-\delta^{\gamma/2}), x]}
	\P(X \in \dd y) \, \left( \sum_{n \in \N} n \, 
	\P\left( S_{n-1} \in x-y+I \right) \right) \,.
\end{split}
\end{equation}

Next we look at the contribution to $\cI_{\delta,x}$
given by $y \in (x-1,x]$.
Applying Lemma~\ref{th:basic},
recalling that $z \mapsto A(z)$ is increasing, for $x-y \le 1$ we have
\begin{equation} \label{eq:analo0}
	\sum_{n \in \N} n \, \P\left( S_{n-1} \in x-y+I \right)
	\le \sum_{n \in \N} \frac{n}{a_{n-1}} \, e^{-c \frac{n-1}{A(1)}} =: C
	< \infty \,,
\end{equation}
hence the contribution to $\cI_{\delta,x}$ in \eqref{eq:2lint}
of $y \in (x-1,x]$ is bounded by
\begin{equation*}
	C \, \int_{y \in (x-1, x]}
	\P(X \in \dd y) = C \, \P(X \in (x-1,x])
	= o\left(\frac{A(x)}{x}\right) \,,
\end{equation*}
where the last equality is a consequence of \eqref{eq:ns12}, see \eqref{eq:nec}.
We can thus rewrite \eqref{eq:2lint} as
\begin{equation}\label{eq:2lint+}
\begin{split}
	\cI_{\delta,x} =
	o\left(\frac{A(x)}{x}\right) \,+\,
	\int_{y \in (x(1-\delta^{\gamma/2}), x-1]}
	\P(X \in \dd y) \, \left( \sum_{n \in \N} n \, 
	\P\left( S_{n-1} \in x-y+I \right) \right) \,.
\end{split}
\end{equation}

Finally, we show in a moment that the following estimate holds:
\begin{equation}\label{eq:clai}
	\sum_{n \in \N} n \, 
	\P\left( S_{n-1} \in w+I \right) 
	\lesssim \frac{A(w)^2}{w} \,, \qquad \forall w \ge 1 \,.
\end{equation}
Plugging this into \eqref{eq:2lint+},
since $x-y \ge 1$, we get
\begin{equation}\label{eq:2lint++}
\begin{split}
	\cI_{\delta,x} & =
	o\left(\frac{A(x)}{x}\right) \,+\,
	\int_{y \in (x(1-\delta^{\gamma/2}), x-1]}
	\P(X \in \dd y) \, \frac{A(x-y)^2}{(x-y)}  \\
	& = o\left(\frac{A(x)}{x}\right) \,+\,
	\int_{s \in [1, \delta^{\gamma/2}x)}
	\frac{A(s)^2}{s} \, \P(X \in x - \dd s) \,,
\end{split}
\end{equation}
by the change of variable $s = x-y$.
Gathering \eqref{eq:dec2}, \eqref{eq:2lint} and \eqref{eq:2lint++},
we have shown that relation \eqref{eq:SRTeq}, and hence
the SRT \eqref{eq:SRT}, holds true for $\alpha \in (\frac{1}{2}, \frac{1}{3}]$.

It only remains to prove \eqref{eq:clai}.
The term $n=1$ contributes only if $0 \in w+I = (w-h,w]$
(recall that $S_0 = 0$),
i.e.\ if $w \le h$. Since $\inf_{w \in [0,h]} A(w)^2/w > 0$, this
gives no problem for \eqref{eq:clai}. For $n \ge 2$ we bound
$n \le 2(n-1)$, and renaming $n-1$ as $m$ we rewrite \eqref{eq:clai} as
\begin{equation} \label{eq:claib}
	\sum_{m \in \N} m \, \P\left( S_{m} \in w+I \right)
	\lesssim \frac{A(w)^2}{w} \,, \qquad \forall w \ge 1 \,.
\end{equation}
\begin{itemize}
\item Let us first look at the contribution of the terms
$m > A(w)$. By Lemma~\ref{th:basic},
\begin{equation*}
	\sum_{m > A(w)} m \, \P\left( S_{m} \in w+I \right)
	\le \sum_{m > A(w)} \frac{m}{a_m} \, e^{-c \frac{m}{A(w)}}
	\le \frac{A(w)^2}{w} \left\{ \sum_{m > A(w)}
	\frac{1}{A(w)} \, \frac{m}{A(w)} \, e^{-c \frac{m}{A(w)}} \right\} \,,
\end{equation*}
because $a_m > w$ for $m > A(w)$.
The bracket is a Riemann sum which converges
to $\int_1^\infty t \, e^{-ct} \, \dd t < \infty$ as $w \to \infty$.
It is also a continuous function of $w$ (by dominated convergence), hence
it is uniformly bounded for $w \in [1,\infty)$. 

\item For
the terms with $m \le A(w)$, we distinguish the events $B_{m,w}^{\ge 1}$
and $B_{m,w}^{0}$, i.e.\ whether there are ``big jumps'' or not
(recall \eqref{eq:B}). Applying
Lemma~\ref{th:kale} with $\delta = 1$, $x=w$ and
with $\ell = m = 1$ (note that $\kappa_\alpha = 1$
and hence $\ell + m \ge \kappa_\alpha +1$), we get
\begin{equation*}
\begin{split}
	\sum_{m \le A(w)} m \, \P\left( S_{m} \in w+I, \, B_{m,w}^{\ge 1} \right) \lesssim
	\frac{A(w)^2}{z} \,.
\end{split}
\end{equation*}
Likewise, by Lemma~\ref{th:nobig} with $\delta = 1$, $x = w$
and $\ell = 1$, we obtain
\begin{equation*}
	\sum_{m \le A(w)} m \, \P\left( S_{m} \in w+I, \, B_{m,w}^{0} \right) 
	\lesssim e^{-1} \frac{A(w)^2}{z}  \,.
\end{equation*}
\end{itemize}
Altogether, we have completed the proof of \eqref{eq:claib},
hence of \eqref{eq:clai}.
\qed

\section{Proof of Theorem~\ref{th:main}: 
sufficiency for $\alpha  \in (0, \frac{1}{2}]$}

\label{sec:gen}

\emph{In this section we assume condition \eqref{eq:ns12}, 
which by Lemma~\ref{lem:tec}
is equivalent to \eqref{eq:ns12'bis0}, and we show that for any
$\alpha \in (0,\frac{1}{2}]$ it implies
\eqref{eq:SRTeq}, which is equivalent to the SRT \eqref{eq:SRT}.}

We stress that the strategy is analogous to the one adopted in Section~\ref{sec:first}
for $\alpha \in (\frac{1}{3}, \frac{1}{2}]$, but
having to deal with more than one big jumps makes things more involved.
In order to keep the exposition as streamlined as possible, we will use a ``backward''
induction, proving the following result, which is stronger than \eqref{eq:SRTeq}.

\begin{theorem}\label{th:main2}
Let $F$ be a probability on $[0,\infty)$
satisfying \eqref{eq:tail1} with $\alpha \in (0, 1)$. 
Assume that condition \eqref{eq:ns12} is satisfied. Then, for every $\ell \in \N_0$,
\begin{equation}\label{eq:ggoal}
	\lim_{\delta \to 0}
	\left( \limsup_{x\to\infty} \frac{x}{A(x)^{\ell+1}}
	\sum_{1 \le n \le A(\delta x)} n^\ell \, \P(S_n \in x+I) \right) = 0 \,.
\end{equation}
In particular, setting $\ell = 0$, relation \eqref{eq:SRTeq} holds.
\end{theorem}

\begin{proof}
Writing $\P(S_n \in x+I) = \P(S_n \in x+I, \, B_{n,x}^{\ge 0} )$,
Lemma~\ref{th:kale} with $m=0$ shows that relation \eqref{eq:ggoal}
holds for all $\ell \ge \kappa_\alpha + 1$. 

We can now proceed by ``backward induction'': we fix
$\bar \ell \in \{0, 1, \ldots, \kappa_\alpha\}$ and assume that \eqref{eq:ggoal} holds for all
$\ell \ge \bar \ell + 1$. If we show that \eqref{eq:ggoal} holds for $\ell = \bar \ell$,
Theorem~\ref{th:main2} is proved.

Let us define $\bar m := \kappa_\alpha - \bar \ell$. Again by Lemma~\ref{th:kale},
for $\delta \le 1$ and $x \ge 1$
\begin{equation*}
	\sum_{1 \le n \le A(\delta x)} n^{\bar \ell}\, \P(S_n \in x+I, \, B_{n,x}^{\ge \bar m + 1})
	\lesssim \delta^\eta \, \frac{A(x)^{\bar \ell + 1}}{x} \,.
\end{equation*}
Likewise, by Lemma~\ref{th:nobig},
\begin{equation*}
	\sum_{1 \le n \le A(\delta x)} n^{\bar \ell}\, \P(S_n \in x+I, \, B_{n,x}^{0})
	\lesssim e^{-\frac{1}{\delta^{\gamma/3}}} \, \frac{A(x)^{\bar \ell + 1}}{x} \,.
\end{equation*}
Therefore, the proof is completed if we show that
for every fixed $m \in \{1,2,\ldots, \bar m\}$
\begin{equation} \label{eq:ifwe}
	\lim_{\delta \to 0}
	\left( \limsup_{x\to\infty} \frac{x}{A(x)^{\bar\ell + 1}}
	\sum_{1 \le n \le A(\delta x)} n^{\bar \ell}\, \P(S_n \in x+I, \, B_{n,x}^{m})
	\right) = 0 \,.
\end{equation}

\medskip

\emph{Proof of \eqref{eq:ifwe}.}
Note that $\P(S_n \in x+I, \, B_{n,x}^{m})=0$ if $n < m$.
For $n \ge m$, plainly,
\begin{equation*}
\begin{split}
	\P\left(S_n \in x+I, \, B_{n,x}^{m}\right)
	& \le n^m \, \P\left( S_n \in x+I, \, \min_{1 \le i \le m}
	X_i > \xi_{n,x}, \, \max_{m+1 \le j \le n} X_j \le \xi_{n,x} \right) \\
	& = n^m \, \int_{(y,w) \in (0, x]^2}
	\P\left(S_m \in \dd y, \, \min_{1 \le i \le m}
	X_i \in \dd w \right) \ind_{\{w > \xi_{n,x}\}} \\
	& \qquad \qquad \qquad \qquad
	\qquad \qquad \P\left(S_{n-m} \in x -y +I, \, B_{n-m,x}^0 \right) \,.
\end{split}
\end{equation*}
Since $w > \xi_{n,x} := a_n^\gamma x^{1-\gamma}$ if and only if
$a_n < (\frac{w}{x})^{1/\gamma} x$, i.e.\ $n <A((\frac{w}{x})^{1/\gamma} x)$,
we obtain
\begin{equation} \label{eq:cheint}
\begin{split}
	& \sum_{1 \le n \le A(\delta x)} n^{\bar \ell}\, \P(S_n \in x+I, \, B_{n,x}^{m}) \\
	& \qquad \qquad \le \int_{(y,w) \in (0, x]^2} \Bigg\{
	\P\left(S_m \in \dd y, \, \min_{1 \le i \le m}
	X_i \in \dd w \right) \\
	& \qquad \qquad \qquad  \qquad \qquad \
	\sum_{m \le n \le A\left(\{(\frac{w}{x})^{1/\gamma} \wedge \delta\} x\right)}
	n^{\bar\ell + m}\,
	\P\left(S_{n-m} \in x -y +I, \, B_{n-m,x}^0 \right) \Bigg\}\,,
\end{split}
\end{equation}
where we set $a \wedge b := \min\{a,b\}$. 
The contribution to the sum of the single term $n=m$ can be bounded as follows:
since $S_{n-m} = S_0 = 0$, by \eqref{eq:nec2}
\begin{equation*}
\begin{split}
	\int_{(y,w) \in (0, x]^2}
	\P\left(S_m \in \dd y, \, \min_{1 \le i \le m}
	X_i \in \dd w \right)
	\ind_{\{0 \in x-y+I\}} 
	& \le \P(S_m \in x+I)
	= o \left(\frac{A(x)}{x}\right) \,,
\end{split}
\end{equation*}
which is negligible for \eqref{eq:ifwe}.
Consequently, we can restrict the sum in \eqref{eq:cheint} to $n \ge m+1$.
In this case $n \le (m+1)(n-m) \lesssim_m (n-m)$, and renaming $n-m$ as $n$ we simplify
\eqref{eq:cheint} as
\begin{equation} \label{eq:cheint2}
\begin{split}
	& \sum_{n \le A(\delta x)} n^{\bar \ell}\, \P(S_n \in x+I, \, B_{n,x}^{m}) \\
	& \qquad \qquad \lesssim_m \int_{(y,w) \in (0, x]^2} \Bigg\{
	\P\left(S_m \in \dd y, \, \min_{1 \le i \le m}
	X_i \in \dd w \right) \\
	& \qquad \qquad \qquad  \qquad \qquad \
	\sum_{1 \le n \le A\left(\{(\frac{w}{x})^{1/\gamma} \wedge \delta\} x\right)}
	n^{\bar\ell + m}\,
	\P\left(S_{n} \in x -y +I, \, B_{n,x}^0 \right) \Bigg\} \,.
\end{split}
\end{equation}
We split the domain of integration in \eqref{eq:cheint2} as 
$(0,x]^2 = J_1 \cup J_2 \cup J_3 \cup J_4$, where
\begin{gather*}
	J_1 := \{y \le x - (\delta^\gamma x \wedge w)\} \,, \qquad
	J_2 := \{y > x-1\} \,, \\
	J_3 := \{w > \delta^\gamma x,\, y \in (x - \delta^\gamma x, x-1]\} \,, 
	\qquad
	J_4 := \{w \le \delta^\gamma x,\, y \in (x - w, x-1]\} \,.
\end{gather*}
and consider each sub-domain separately.

\medskip

\emph{Contribution of $J_1$}.
Let us set
\begin{equation}\label{eq:hatdelta}
	\hat\delta = \hat\delta(w,x,\delta) := 
	\left(\frac{w}{x}\right)^{1/\gamma} \wedge \delta \,,
\end{equation}
so that $J_1 = \{y  \le x - \hat\delta^\gamma x\}$. 
Since $x-y \ge \hat\delta^\gamma x$ on $J_1$,
the sum in \eqref{eq:cheint2} is bounded by
\begin{equation*}
	\sum_{1 \le n \le A\left(\hat\delta x\right)}
	n^{\bar\ell + m}\,
	\sup_{z \ge \hat\delta^{\gamma} x}
	\P\left(S_{n} \in z +I, \, B_{n,x}^0 \right)
	\lesssim e^{-\frac{1}{\hat\delta^{\gamma/3}}} \, 
	\frac{A(x)^{\bar\ell+m+1}}{x} \,,
\end{equation*}
where the inequality follows by Lemma~\ref{th:nobig}, with $\delta$ replaced by
$\hat\delta$ and $\ell$ replaced by $\bar\ell + m$.
The contribution of $J_1$
to the integral in \eqref{eq:cheint2} is thus bounded by
\begin{equation}\label{eq:con}
\begin{split}
	& \lesssim \frac{A(x)^{\bar\ell+m+1}}{x}
	\int_{w \in (0, x], \ y \in (0, x - (\delta^\gamma x \wedge w)]}
	\P\left(S_m \in \dd y, \, \min_{1 \le i \le m}
	X_i \in \dd w \right)  \, e^{-\frac{1}{(\frac{w}{x})^{1/3} \wedge \delta^{\gamma/3}}} \,.
\end{split}
\end{equation}

We split this integral in the sub-domains 
$J_1^\le := \{w \le \delta^\gamma x\}$
and $J_1^> := \{w > \delta^\gamma x\}$. Bounding
$\P(X > \delta^\gamma x) \lesssim 1/A(\delta^\gamma x) 
\lesssim \delta^{-2\gamma\alpha}/A(x)$,
by the lower bound in \eqref{eq:Potter} with $\epsilon = \alpha$ and $\rho = \delta$,
the contribution of $J_1^>$ is controlled by
\begin{equation*}
\begin{split}
	\frac{A(x)^{\bar\ell+m+1}}{x} \, e^{-\frac{1}{\delta^{\gamma/3}}}
	\P\left(\min_{1 \le i \le m} X_i > \delta^\gamma x\right) & =
	e^{-\frac{1}{\delta^{\gamma/3}}} \, \frac{A(x)^{\bar\ell+m+1}}{x} \, 
	\P\left(X > \delta^\gamma x\right)^m \\
	& \lesssim \frac{e^{-\frac{1}{\delta^{\gamma/3}}}}{\delta^{2\gamma \alpha m}} \,
	\frac{A(x)^{\bar\ell+1}}{x} \,,
\end{split}
\end{equation*}
which gives no problem for \eqref{eq:ifwe}. Next we bound the contribution of 
$J_1^\le$ to \eqref{eq:con} by
\begin{equation*}
	\frac{A(x)^{\bar\ell+m+1}}{x} \,
	\int_{w \in (0, \delta^\gamma x]} \P\left(\min_{1 \le i \le m}
	X_i \in \dd w \right)  \, \phi\left(\frac{w}{x}\right) \,,
	\qquad \text{with} \qquad \phi(t) := e^{-\frac{1}{t^{1/3}}} \,.
\end{equation*}
We set $G(w) := \P\left(\min_{1 \le i \le m} X_i > w \right)$, so that
$\P\left(\min_{1 \le i \le m} X_i \in \dd w \right) = -\dd G(w)$.
Integrating by parts, since the contribution of the boundary terms is negative, we get
\begin{equation*}
	\frac{A(x)^{\bar\ell+m+1}}{x} \,
	\int_{0}^{\delta^\gamma x} G(w)
	\, \phi'\left(\frac{w}{x}\right) \, \frac{1}{x}
	\, \dd w
	\lesssim \frac{A(x)^{\bar\ell+m+1}}{x} \,
	\int_{0}^{\delta^\gamma x} \frac{1}{A(w)^m}
	\, \phi'\left(\frac{w}{x}\right) \, \frac{1}{x}
	\, \dd w \,.
\end{equation*}
Performing the change of variable $v = w/x$,
since $A(vx) \gtrsim A(x) v^{2\alpha}$ 
by \eqref{eq:Potter}, we obtain
\begin{equation*}
	\lesssim \frac{A(x)^{\bar\ell+1}}{x}
	\int_{0}^{\delta^\gamma} \frac{\phi'\left(v\right)}{v^{2\alpha m}}
	\, \dd v = \frac{A(x)^{\bar\ell+1}}{x}
	\int_{0}^{\delta^\gamma} \frac{e^{-\frac{1}{v^{1/3}}}}
	{3 \, v^{2\alpha m + 4/3}}
	\, \dd v \lesssim
	\frac{A(x)^{\bar\ell+1}}{x}
	\int_{0}^{\delta^\gamma} e^{-\frac{1}{2 v^{1/3}}}
	\, \dd v \,,
\end{equation*}
which again gives no problem for \eqref{eq:ifwe}.
Overall, the contribution of $J_1$ is under control.

\medskip

\emph{Contribution of $J_2$.}
By Lemma~\ref{th:basic}, for $x-y \le 1$ we have
\begin{equation} \label{eq:analo1}
	\sum_{n \in \N} n^{\bar \ell + m} \, \P\left( S_{n-1} \in x-y+I \right)
	\le \sum_{n \in \N} \frac{n^{\bar \ell + m}}{a_{n-1}} \, e^{-c \frac{n-1}{A(1)}} =: 
	C_{\bar \ell + m}
	< \infty \,,
\end{equation}
because $z \mapsto A(z)$ is increasing,
hence the contribution of $J_2$ to \eqref{eq:cheint2} is bounded by
\begin{equation*}
	\int_{(y,w) \in J_2} \P\left(S_m \in \dd y, \, \min_{1 \le i \le m}
	X_i \in \dd w \right)
	C_{\bar\ell + m} \lesssim_{\bar\ell,m} \P(S_m \in (x-1,x])
	= o\left(\frac{A(x)}{x}\right) \,,
\end{equation*}
where the last equality is a consequence of \eqref{eq:ns12}, see \eqref{eq:nec2}.
This shows that $J_2$ gives a negligible contribution to \eqref{eq:ifwe}.

\medskip

\emph{Technical interlude.}
Before analyzing $J_3$ and $J_4$, let us elaborate on \eqref{eq:ggoal}
(where we rename $\ell$ as $k$ and $x$ as $z$ for later convenience).
Our induction hypothesis that \eqref{eq:ggoal} holds for all $k \ge \bar \ell + 1$
can be rewritten as follows:
for every $\delta > 0$ there is $\bar z_k(\delta) < \infty$ such that
\begin{equation}\label{eq:rewri}
	\sum_{1 \le n \le A(\delta z)} n^k \, \P(S_n \in z+I)
	\le f_k(\delta) \, \frac{A(z)^{k+1}}{z}\,, \qquad \forall
	k \ge \bar\ell + 1, \ \forall z \ge \bar z_k(\delta) \,,
\end{equation}
where we set $f_k(\delta) := 2 \limsup_{x\to\infty} (\ldots)$
in \eqref{eq:ggoal} (with $\ell$ replaced by $k$), so that
\begin{equation}\label{eq:limit}
	\lim_{\delta \to 0} f_k(\delta) = 0 \,.
\end{equation}
We also claim that
\begin{equation}\label{eq:rewricon}
	\sum_{n\in\N} n^k \, \P(S_n \in z+I)
	\lesssim_k \frac{A(z)^{k+1}}{z}\,, \qquad \forall
	k \ge \bar\ell + 1, \ \forall z \ge 1 \,.
\end{equation}
To show this, fix $\bar\delta_k \in (0,1]$ such that $f_k(\bar\delta_k) \le 1$, 
by \eqref{eq:limit}.
If we restrict the sum to $n \le A(\bar\delta_k z)$,
relation \eqref{eq:rewri} shows that \eqref{eq:rewricon} holds
for $z \ge \bar z_k(\bar\delta)$, while for $z \le \bar z_k(\bar\delta)$
\begin{equation*}
	\sum_{n\le A(\bar\delta_k z)} n^k \, \P(S_n \in z+I) \le \sum_{n\le A( z)} n^k
	\le A(z)^{k+1} \le \bar z_k(\bar\delta_k) \,
	\frac{A(z)^{k+1}}{z} \lesssim_k \frac{A(z)^{k+1}}{z} \,.
\end{equation*}
It remains to prove that \eqref{eq:rewricon} holds for the sum restricted
to the terms with $n > A(\bar\delta_k z)$: applying \eqref{eq:basic}
followed by $a_n \ge a_{A(\bar\delta_k z)} = \bar\delta_k z \gtrsim_k z$, we can write
\begin{equation*}
	\sum_{n > A(\bar\delta_k z)} n^k \, \P(S_n \in z+I) \lesssim
	\sum_{n > A(\bar\delta_k z)} \frac{n^k}{a_n} \, e^{-c \frac{n}{A(z)}}
	\lesssim_k \frac{A(z)^{k+1}}{z} \, \left\{
	\sum_{n \in\N}\frac{1}{A(z)} 
	\left(\frac{n}{A(z)}\right)^k \, e^{-c \frac{n}{A(z)}} \right\} \,.
\end{equation*}
The bracket is a Riemann sum 
which converges to the integral 
$\int_0^\infty t^k \, e^{-ct} \, \dd t < \infty$ as $z \to \infty$.
Being a continuous function of $z$ (by dominated convergence),
the sum is uniformly bounded for $z \in [1,\infty)$.
The proof of \eqref{eq:rewricon} is completed.

Let us finally rewrite \eqref{eq:ns12}, which is equivalent to our assumption \eqref{eq:ns12'},
as follows: defining $g(\eta) := 2 \, \limsup_{x\to\infty} (\ldots)$ in \eqref{eq:ns12},
for every $\eta \in (0,1]$ there is $\tilde z(\eta) < \infty$ such that
\begin{equation}\label{eq:ns12new}
	\int_{s \in [1,\eta z)} 
	\frac{A(s)^2}{s} \, \P(X \in z - \dd s) \le
	g(\eta) \, \frac{A(z)}{z}  \,, \qquad \forall z \ge \tilde z(\eta) \,,
\end{equation}
with
\begin{equation}\label{eq:moreover}
	\lim_{\eta \to 0} g(\eta) = 0 \,.
\end{equation}
Moreover, we claim that for any $\zeta \in (0,1)$
\begin{equation}\label{eq:ns12newcon}
	\int_{s \in [1, \zeta z]} 
	\frac{A(s)^2}{s} \, \P(X \in z - \dd s) \lesssim_\zeta
	\frac{A(z)}{z}  \,, \qquad \forall z \ge 1 \,.
\end{equation}
To show this, let us fix $\bar\eta \in (0, 1)$ such that $g(\bar\eta) \le 1$, 
and split $\int_{s \in [1, \zeta z]} = \int_{s \in [1,\bar\eta z)} 
+ \int_{s \in [\bar\eta z , \zeta z]}$.
The contribution of $[1,\bar\eta z)$ is controlled by relation \eqref{eq:ns12new}
for $z \ge \tilde z(\bar\eta)$, while for $z < \tilde z(\bar\eta)$
it is enough to note that $c := \inf_{z \in [1, \tilde z(\bar\eta)]} \frac{A(z)}{z} > 0$
while
\begin{equation*}
	\sup_{z \in [1, \tilde z(\bar\eta)]} \int_{s \in [1,\bar\eta z)} 
	\frac{A(s)^2}{s} \, \P(X \in z - \dd s) \le A(\tilde z(\bar\eta))^2
	=: C < \infty \,,
\end{equation*}
hence \eqref{eq:ns12newcon} holds restricted to $[1,\bar\eta z)$.
Finally, for the integral over $[\bar\eta z , \zeta z]$
we estimate
\begin{equation*}
	\int_{s \in [\bar\eta z, \zeta z]} 
	\frac{A(s)^2}{s} \, \P(X \in z - \dd s) \le
	\frac{A(z)^2}{\bar\eta \, z} \P(X \ge (1-\zeta) z)
	\lesssim_\eta \frac{A(z)}{z} \,,
\end{equation*}
completing the proof of \eqref{eq:ns12newcon}.

\medskip

\emph{Contribution of $J_3$.}
We recall that
\begin{equation*}
	J_3 = \{w > \delta^\gamma x, \, y \in (x-\delta^\gamma x, x-1]\} \,.
\end{equation*}
For $m=1$, since
$S_m = \min_{1 \le i \le m} X_i = X_1$,
we have $J_3 = \{y \in (x-\delta^\gamma x, x-1], \, w=y\}$.
Applying \eqref{eq:rewricon} with
$k = \bar\ell+1$ and $z=x-y$, the contribution of $J_3$
to \eqref{eq:cheint2} is bounded by
\begin{equation*}
	\lesssim_{\bar\ell} \int_{y \in (x-\delta^\gamma x, x-1]} \P(X \in \dd y) \,
	\frac{A(x-y)^{\bar \ell + 2}}{x-y} \le
	A(x)^{\bar \ell}
	\int_{s \in [1, \delta^\gamma x)} \P(X \in x - \dd s) \,
	\frac{A(s)^2}{s} \,,
\end{equation*}
where we have performed the change of variable $s = x-y$.
Applying \eqref{eq:ns12}, or equivalently \eqref{eq:ns12new}-\eqref{eq:moreover},
it follows immediately that $J_3$ gives no problem for \eqref{eq:ifwe},
when $m=1$.

Next we assume that $m \ge 2$. It is convenient to set
\begin{equation}\label{eq:LambdaM}
	\Lambda_m := \min_{1 \le i \le m} X_i \,, \qquad
	M_m := \max_{1 \le i \le m} X_i \,.
\end{equation}
Applying \eqref{eq:rewricon} for $k = \bar \ell + m$, the contribution of 
$J_3$ to \eqref{eq:cheint2} is bounded by
\begin{equation} \label{eq:trel0}
	\lesssim_{\bar\ell, m} \int_{y \in (x-\delta^\gamma x, x-1]} \Bigg\{
	\P\left(S_m \in \dd y, \, \Lambda_m > \delta^\gamma x \right)
	\frac{A(x-y)^{\bar \ell + m + 1}}{x-y} \Bigg\} \,.
\end{equation}

We need to estimate $\P\left(S_m \in \dd y, \, \Lambda_m > \delta^\gamma x \right)$.
The events $\{X_j \ge \max_{i \in \{1,\ldots, m\} \setminus \{j\}} X_i\}$
for $j=1,\ldots, m$ cover the whole probability space and have the same probability,
hence
\begin{equation} \label{eq:doli}
\begin{split}
	& \P\left(S_m \in \dd y, \, \Lambda_m \in \dd w \right) \le m \,
	\P\left(S_m \in \dd y, \, \Lambda_{m} \in \dd w, \, 
	M_{m-1} \le X_m \right) \\
	& \ \ \le m
	\int_{u,v \in (0, y]}
	\P\left(S_{m-1} \in \dd u, \, \Lambda_{m-1} \in \dd w ,
	\, M_{m-1} \in \dd v\right) \ind_{\{v \le y-u\}} \,
	\P(X \in \dd y - u) \,,
\end{split}
\end{equation}
where $\ind_{\{v \le y-u\}}$ comes from $\{M_{m-1} \le X_m\}$.
Note that
\begin{equation*}
	u = S_{m-1} \le (m-1) M_{m-1} = (m-1) v \le (m-1)(y-u) \,,
\end{equation*}
which yields the restriction $u \le \frac{m-1}{m}y$.
In particular, for $y \le x$ we have $y \le \frac{m-1}{m}x$,
which by \eqref{eq:doli} yields the bound
\begin{equation}\label{eq:ufff}
\begin{split}
	& \P\left(S_m \in \dd y, \, \Lambda_m \in \dd w \right) \\
	& \qquad \le m 
	\int_{u \in (0, \frac{m-1}{m}x]}
	\P\left(S_{m-1} \in \dd u, \, \Lambda_{m-1} \in \dd w \right)
	\P(X \in \dd y - u) \,.
\end{split}
\end{equation}
Plugging this into \eqref{eq:trel0}, the contribution of $J_3$
to \eqref{eq:cheint2} is bounded by
\begin{equation} \label{eq:trel}
\begin{split}
	& \qquad \lesssim_{m} 
	\int_{u \in (0, \frac{m-1}{m}x]}
	\P\left(S_{m-1} \in \dd u, \, \Lambda_{m-1} > \delta^\gamma x \right) \\
	& \qquad \qquad \qquad \qquad \qquad \qquad \Bigg\{
	\int_{y \in (x-\delta^\gamma x, x-1]} \P(X \in \dd y - u)
	\frac{A(x-y)^{\bar \ell + m + 1}}{x-y} \Bigg\} \,.
\end{split}
\end{equation}

With the change of variables $s = x-y$, the term in bracket 
in \eqref{eq:trel} becomes
\begin{equation*}
\begin{split}
	\int_{s \in [1, \delta^\gamma x)} \P(X \in x - u - \dd s)
	& \frac{A(s)^{\bar \ell + m + 1}}{s} 
	\le A(\delta^\gamma x)^{\bar \ell + m - 1} \,
	\int_{s \in [1, \delta^\gamma x)} \P(X \in x - u - \dd s)
	\frac{A(s)^{2}}{s} \\
	& \le A(\delta^\gamma x)^{\bar \ell + m - 1} \,
	\int_{s \in [1, m\delta^\gamma (x-u))} \P(X \in x - u - \dd s)
	\frac{A(s)^{2}}{s} \,,
\end{split}
\end{equation*}
where in the last inequality we have enlarged the domain of integration,
for $u \le \frac{m-1}{m} x$ (as in \eqref{eq:trel}).
Since $x-u \ge \frac{1}{m} x$, we can apply \eqref{eq:ns12new} 
with $z = x-u$ and $\eta = m \delta^\gamma$,
provided $x$ is large enough
(so that $\frac{1}{m} x \ge \tilde z(m \delta^\gamma)$). This allows to bound \eqref{eq:trel} by
\begin{equation*}
\begin{split}
	& \le A(\delta^\gamma x)^{\bar \ell + m - 1} 
	\int_{u \in (0, \frac{m-1}{m}x]}
	\P\left(S_{m-1} \in \dd u, \, \Lambda_{m-1} > \delta^\gamma x \right) \,
	\left\{ g(m \delta^\gamma) \, \frac{A(x-u)}{x-u} \right\} \\
	& \le A(\delta^\gamma x)^{\bar \ell + m - 1} 
	\P\left(\Lambda_{m-1} > \delta^\gamma x \right) \,
	\, g(m \delta^\gamma) \, \frac{A(x)}{\frac{1}{m}x}  \,,
\end{split}	
\end{equation*}
and since $\P\left(\Lambda_{m-1} > t \right) = \P(X >t)^{m-1}
\sim 1/A(t)^{m-1}$ the last line is
\begin{equation*}
	\sim A(\delta^\gamma x)^{\bar \ell} 
	\, g(m \delta^\gamma) \, \frac{m \, A(x)}{x} \\
	\lesssim_m g(m \delta^\gamma) \,
	\frac{A(x)^{\bar\ell + 1}}{x} \,.
\end{equation*}
Plugging this bound into \eqref{eq:ifwe}
and applying \eqref{eq:moreover}, we have shown that the contribution
of $J_3$ is under control.

\medskip

\emph{Contribution of $J_4$.}
Note that $J_4 := \{w \le \delta^\gamma x, \, y \in (x - w, x-1]\}$
is empty for $m=1$, provided $\delta > 0$ is small enough:
in fact, relations $y > x-w$ and $w \le \delta^\gamma x$
cannot be fulfilled simultaneously, since $y=w$ for $m=1$.
Henceforth we assume that $m\ge 2$.

Recalling \eqref{eq:LambdaM} and
plugging \eqref{eq:rewricon}
with $k = \bar \ell + m$ into \eqref{eq:cheint2}, 
the contribution of $J_4$ is bounded as follows:
\begin{equation} \label{eq:trel2B}
\begin{split}
	& \lesssim_{\bar\ell, m} \int_{w \in (1, \delta^\gamma x], \, y \in (x - w, x-1]} 
	\P\left(S_m \in \dd y, \, \Lambda_m \in \dd w \right)
	\frac{A(x-y)^{\bar \ell + m + 1}}{x-y} \,.
\end{split}
\end{equation}
Our goal is to show that this satisfies \eqref{eq:ifwe}.
It is convenient to set for $C, D \in (0,\infty)$
\begin{equation}\label{eq:Theta}
\begin{split}
	\Theta_{\bar\ell, m}^{C,D}(x,\delta) & :=
	\int_{w \in [C, \delta^\gamma x], \, y \in [x - Dw, x-1]} 
	\P\left(S_m \in \dd y, \, \Lambda_m \in \dd w \right)
	\frac{A(x-y)^{\bar \ell + m + 1}}{x-y} \,,
\end{split}	
\end{equation}
so that \eqref{eq:trel2B} is bounded from above
by $\Theta_{\bar\ell, m}^{C,D}(x,\delta)$
with $C=D=1$. Consequently, to prove our goal \eqref{eq:ifwe} it is enough to show 
the following: recalling that $\bar \ell \in \{0, \ldots, \kappa_\alpha\}$ is fixed,
\begin{equation}\label{eq:ifwelast}
	\lim_{\delta \to 0} \left( \limsup_{x\to\infty}
	\frac{x}{A(x)^{\bar\ell + 1}} 
	\Theta_{\bar\ell, m}^{C,D}(x,\delta) \right) = 0 \,, \qquad
	\forall C, D \in (0,\infty) \,, \ \
	\forall m \in \{1,2,\ldots, \bar m\}\,.
\end{equation}

\smallskip

Note that
\begin{equation*}
\begin{split}
	\P\left(S_m \in \dd y, \, \Lambda_m \in \dd w \right) 
	& \le m \,
	\P\left(S_m \in \dd y, \, X_m \in \dd w, \, 
	\Lambda_{m-1} \ge w \right) \\
	& = m \,
	\P(X \in \dd w) \, \P\left( S_{m-1} \in \dd y - w, \,
	\Lambda_{m-1} \ge w \right) \,,
\end{split}
\end{equation*}
therefore
\begin{equation*}
\begin{split}
	& \Theta^{C,D}_{\bar\ell, m}(x,\delta) 
	\lesssim_m 
	\int_{w \in [C, \delta^\gamma x]} \P(X \in \dd w) \\
	& \qquad \qquad \qquad \qquad \Bigg\{
	\int_{y \in [x-Dw, x-1]}
	\P\left(S_{m-1} \in \dd y - w, \, \Lambda_{m-1} \ge w \right)  
	\,\frac{A(x-y)^{\bar \ell + m + 1}}{x-y} \Bigg\} \,.
\end{split}
\end{equation*}
Next we change variable from $y$ to $s = (w + x) - y$ in the inner integral
(for fixed $w$).
Since $\dd y - w = x - \dd s$ and $x-y = s-w$, we get
\begin{equation*}
\begin{split}
	& \int_{w \in [C,\delta^\gamma x]} \P(X \in \dd w)
	\Bigg\{ \int_{s \in [1+w, (D+1)w]}
	\P\left(S_{m-1} \in x - \dd s, \, \Lambda_{m-1} \ge w \right)  
	\frac{A(s-w)^{\bar \ell + m + 1}}{s-w} \Bigg\} \,.
\end{split}
\end{equation*}
Writing $\P\left(S_{m-1} \in x - \dd s, \, \Lambda_{m-1} \ge w \right) 
= \int_{u \in [0, \infty)} \P\left(S_{m-1} \in x - \dd s, \, \Lambda_{m-1} \in \dd u \right)
\, \ind_{\{u \ge w\}}$
and observing that
\begin{equation*}
	\{w \in [C,\delta^\gamma x], \, s \in [1+w, (1+D)w]\}
	\subseteq \{s \in [1+C, (1+D)\delta^\gamma x], \, w \in [\tfrac{s}{1+D}, s-1]\} \,,
\end{equation*}
we obtain by Fubini's theorem
\begin{equation*}
\begin{split}
	\Theta^{C,D}_{\bar\ell, m}(x,\delta) \lesssim_m 
	& \int_{s \in [1+C, (1+D)\delta^\gamma x], \, u \in [0,\infty)} 
	\P\left(S_{m-1} \in x - \dd s, \, \Lambda_{m-1} \in \dd u \right) \\
	& \qquad \qquad \qquad
	\Bigg\{ \int_{w \in [\frac{s}{1+D}, s-1]}
	\P(X \in \dd w) \,
	\frac{A(s-w)^{\bar \ell + m + 1}}{s-w} \,
	\ind_{\{w \le u\}} \Bigg\} \,.
\end{split}
\end{equation*}
We can restrict the domain of integration for $u$ to
$[\frac{s}{1+D}, \infty)$, because for $u < \frac{s}{1+D}$ the inner
integral vanishes, due to $\ind_{\{w \le u\}}$.
After this restriction, we drop $\ind_{\{w \le u\}}$ and
change variable from $w$ to $t = s-w$ in the inner integral, getting
\begin{equation}\label{eq:atlas}
\begin{split}
	\Theta^{C,D}_{\bar\ell, m}(x,\delta) \lesssim_m 
	& \int_{s \in [1+C, (1+D)\delta^\gamma x], \, u \in [\frac{s}{1+D},\infty)} 
	\P\left(S_{m-1} \in x - \dd s, \, \Lambda_{m-1} \in \dd u \right) \\
	& \qquad \qquad \qquad
	\Bigg\{ \int_{t \in [1, \frac{Ds}{1+D}]}
	\P(X \in s - \dd t) \,
	\frac{A(t)^{\bar \ell + m + 1}}{t} \Bigg\} \,.
\end{split}
\end{equation}
Applying \eqref{eq:ns12newcon} with $z=s$ and $\zeta = \frac{D}{1+D}$ 
allows to bound the term in bracket by
\begin{equation}\label{eq:atlas2}
\begin{split}
	A(s)^{\bar \ell + m - 1}
	\int_{t \in [1, \frac{Ds}{1+D}]}
	\P(X \in s - \dd t) \,
	\frac{A(t)^{2}}{t} \lesssim_D \frac{A(s)^{\bar \ell + m}}{s} \,,
\end{split}
\end{equation}
hence from \eqref{eq:atlas} we get the crucial estimate
\begin{equation} \label{eq:cruind}
\begin{split}
	\Theta^{C,D}_{\bar\ell, m}(x,\delta) 
	& \lesssim_m  \int_{s \in [1+C, 
	(1+D)\delta^\gamma x], \, u \in [\frac{s}{1+D},\infty)} 
	\P\left(S_{m-1} \in x - \dd s, \, \Lambda_{m-1} \in \dd u \right)
	\frac{A(s)^{\bar \ell + m}}{s} \,.
\end{split}
\end{equation}

Let us first consider the case $m = 2$.
Then $S_{m-1} = X_1$, hence by \eqref{eq:ns12new} with $z=x$ and
$\eta = (1+D)\delta^\gamma$ we get
\begin{equation*}
	\Theta^{C,D}_{\bar\ell, 2}(x,\delta) 
	\lesssim_m A\left(x\right)^{\bar \ell} 
	\int_{s \in [1+C, (1+D)\delta^\gamma x]} 
	\P\left(X \in x - \dd s \right)
	\frac{A(s)^{2}}{s} \lesssim g\left((1+D)\delta^\gamma\right)
	\, \frac{A(x)^{\bar\ell+1}}{x} \,,
\end{equation*}
and recalling \eqref{eq:moreover} it follows that \eqref{eq:ifwelast} is proved.

Henceforth we assume that $m \ge 3$.
We start focusing on the contribution to \eqref{eq:cruind}
given by $u \ge \delta^\gamma x$, which is bounded by
\begin{equation*}
	\int_{s \in [1+C, 
	(1+D)\delta^\gamma x]} 
	\P\left(S_{m-1} \in x - \dd s, \, \Lambda_{m-1} \ge \delta^\gamma x \right)
	\frac{A(s)^{\bar \ell + m}}{s} \,,
\end{equation*}
and applying \eqref{eq:ufff} with $m$ replaced by $m-1$ we get, by Fubini's theorem,
\begin{equation} \label{eq:dela}
\begin{split}
	& \lesssim_m \int_{s \in [1+C, (1+D)\delta^\gamma x]} 
	\frac{A(s)^{\bar \ell + m}}{s} \\
	& \qquad \qquad \qquad \left\{ \int_{u \in (0, \frac{m-2}{m-1} x]}
	\P\left(S_{m-2} \in \dd u, \, \Lambda_{m-2} \ge \delta^\gamma x \right)
	\, \P( X \in x - u - \dd s) \right\} \\
	& \le A((1+D)\delta^\gamma x)^{\bar\ell + m -2} \int_{u \in (0, \frac{m-2}{m-1} x]}
	\P\left(S_{m-2} \in \dd u, \, \Lambda_{m-2} \ge \delta^\gamma x \right) \\
	& \qquad \qquad \qquad \qquad \qquad \left\{ 
	\int_{s \in [1+C, (1+D)\delta^\gamma x]} \frac{A(s)^{2}}{s}
	\P( X \in x - u - \dd s) \right\} \,.
\end{split}
\end{equation}
Concerning the inner integral, we enlarge the domain of integration
to $[1, \hat \eta (x-u))$ with
\begin{equation} \label{eq:bareta}
	\hat \eta = \hat\eta_{m,D,\delta} 
	:= 2(1+D)\delta^\gamma \sup_{u \le \frac{m-2}{m-1}x} \frac{x}{x-u}
	= 2(m-1) (1+D)\delta^\gamma \,,
\end{equation}
after which we can apply \eqref{eq:ns12new} with $z = x-u$
and $\eta = \hat \eta$
(which satisfies $z \ge \tilde z(\hat \eta)$ provided $x$ is large enough,
since $x-u \ge \frac{x}{m-1}$). In this way, 
\begin{equation*}
	\left\{ \int_{s \in [1+C, (1+D)\delta^\gamma x]} \frac{A(s)^{2}}{s}
	\P( X \in x - u - \dd s) \right\} \le
	g(\hat\eta) \, \frac{A(x-u)}{x-u} \lesssim_m
	g(\hat\eta) \, \frac{A(x)}{x} \,,
\end{equation*}
where the last inequality holds again because $x-u \ge \frac{x}{m-1}$
(recall that $x \mapsto A(x)/x$ is regularly varying
with index $\alpha-1 < 0$). Then
\eqref{eq:dela} is bounded by
\begin{equation*}
	\lesssim_{D} 
	A(\delta^\gamma x)^{\bar\ell + m -2} \,
	g(\hat\eta) \,\frac{A(x)}{x} \, \P\left(\Lambda_{m-2} \ge \delta^\gamma x \right)
	\lesssim g(\hat\eta) \, \frac{A(x)^{\bar\ell + 1}}{x} \,,
\end{equation*}
because $\P(\Lambda_{m-2} \ge t) = \P(X \ge t)^{m-2} \sim 1/A(t)^{m-2}$.
Looking back at our goal \eqref{eq:ifwelast}, and recalling \eqref{eq:bareta}
and \eqref{eq:moreover},
the contribution of $u \ge \delta^\gamma x$ to \eqref{eq:cruind} is under control.

It finally remains to consider the contribution of $u < \delta^\gamma x$ to \eqref{eq:cruind}:
since
\begin{equation*}
	\{s \in [1+C, (1+D)\delta^\gamma x], \, u \in [\tfrac{s}{1+D}, \delta^\gamma x]\}
	= \{u \in [\tfrac{1+C}{1+D}, \delta^\gamma x], \,
	s \in [1+C, (1+D)u] \} \,,
\end{equation*}
applying Fubini's theorem we can write such a contribution as follows:
\begin{equation*}
\begin{split}
	& \int_{u \in [\tfrac{1+C}{1+D}, \delta^\gamma x], \,
	s \in [1+C, (1+D)u]} 
	\P\left(S_{m-1} \in x - \dd s, \, \Lambda_{m-1} \in \dd u \right)
	\frac{A(s)^{\bar \ell + m}}{s} \\
	& = \int_{u \in [\tfrac{1+C}{1+D}, \delta^\gamma x], \,
	y \in [x- (1+D)u, x-(1+C)]} 
	\P\left(S_{m-1} \in \dd y, \, \Lambda_{m-1} \in \dd u \right)
	\frac{A(x-y)^{\bar \ell + m}}{x-y} \\
	& \le \Theta^{C',D'}_{\bar\ell, m-1}(x,\delta) \,, \qquad
	\text{with} \qquad C' := \tfrac{1+C}{1+D} \,, \ D' := 1+D
	\,,
\end{split}
\end{equation*}
where for the last inequality we recall \eqref{eq:Theta}. Therefore
\begin{equation} \label{eq:conclude}
	\lim_{\delta \to 0} \left( \limsup_{x\to\infty}
	\frac{x}{A(x)^{\bar\ell + 1}} 
	\Theta_{\bar\ell, m}^{C,D}(x,\delta) \right) \le
	\lim_{\delta \to 0} \left( \limsup_{x\to\infty}
	\frac{x}{A(x)^{\bar\ell + 1}} 
	\Theta_{\bar\ell, m-1}^{C',D'}(x,\delta) \right) \,.
\end{equation}
We can then conclude by induction on $m$. In fact, we have already proved that \eqref{eq:ifwelast}
holds for $m=2$, and relation \eqref{eq:conclude} shows that if it holds for
$m-1$ then it holds for $m$.
\end{proof}

\section{Proof of Proposition~\ref{pr:main1} and
and of Theorems~\ref{th:1/2} and~\ref{th:cont}}
\label{sec:prmain}

\subsection{Proof of Proposition~\ref{pr:main1}}\label{sec:suff}

We can reformulate condition \eqref{eq:suff0} equivalently as follows:
there exist $x_0, C \in (0,\infty)$ such that
(for the same $\epsilon > 0$ as in \eqref{eq:suff0})
\begin{equation} \label{eq:suff}
	\frac{F((x,x+s])}{F((x,\infty))}
	\le C \left(\frac{s}{x}\right)^{1-2\alpha+\epsilon}  ,
	\qquad \ \forall x \ge x_0 \,, \ \forall s \in [1,x] \,.
\end{equation}
It is clear that \eqref{eq:suff} implies \eqref{eq:suff0},
and the converse also holds, by a contradiction argument.

Then it suffices to show that condition \eqref{eq:suff} implies \eqref{eq:ns}
(which is equivalent to \eqref{eq:nsbis0}, by Lemma~\ref{lem:tec}).
For $x \ge 2x_0$ and $0 \le s \le \frac{1}{2}x$, by \eqref{eq:suff},
\begin{equation} \label{eq:analo}
	\P(X \in (x-s,x]) \le C \, \P(X > x-s) \left(\frac{s}{x}\right)^{1-2\alpha+\epsilon}
	\lesssim \frac{1}{A( x)} \left(\frac{s}{x}\right)^{1-2\alpha+\epsilon} \,.
\end{equation}
Since $\frac{A(s)^2}{s^2} \, s^{1-2\alpha+\epsilon}$ is regularly varying with index 
$(2\alpha - 2) +1-2\alpha+\epsilon = -1+\epsilon > -1$,
one has $\int_1^{z} \frac{A(s)^2}{s^2} \, s^{1-2\alpha+\epsilon} \, \dd s \lesssim
A(z)^2\, z^{-2\alpha+\epsilon}$ by \cite[Proposition~1.5.8]{cf:BinGolTeu}, 
hence
\begin{equation*}
	\frac{x}{A(x)} \int_1^{\eta x} \frac{A(s)^2}{s^2} \, \P(X \in (x-s,x]) \, \dd s
	\lesssim \frac{A(\eta x)^2 \, 
	(\eta x)^{-2\alpha+\epsilon}}{A(x)^2 \, x^{-2\alpha+\epsilon}}
	\xrightarrow[x\to\infty]{} \eta^{\epsilon} \,.
\end{equation*}
Then \eqref{eq:ns} follows.\qed

\subsection{Proof of Theorem~\ref{th:1/2}}

We recall that $A(x) = L(x) \sqrt{x}$ with $L \in \cR_0$,
and a sufficient condition for the SRT \eqref{eq:SRT}
when $\alpha = \frac{1}{2}$ is given by \eqref{eq:suffcond1/2}.

If \eqref{eq:cond} holds, we can write 
$L^*(x) \lesssim L(x) = A(x)/\sqrt{x}$, hence \eqref{eq:suffcond1/2} is implied by
\begin{equation}\label{eq:suffcond1/2red}
	\exists T \in [0,\infty): \qquad
	\lim_{\eta \to 0} \left(
	\limsup_{x\to\infty} \,
	\frac{\RR_T\big( (1-\eta)x,x \big)}{x} \right) = 0 \,.
\end{equation}
It is easy to show that this holds for $T=0$,
with no extra assumption on $F$.
By \eqref{eq:rr}-\eqref{eq:RR}
\begin{equation} \label{eq:male}
	\RR_0\big( (1-\eta)x,x \big) =
	\int_{(1-\eta)x}^x y \, A(y) \, F(y+I) \, \dd y
	\le x \, A(x) \int_{(1-\eta)x}^x  F(y+I) \, \dd y \,,
\end{equation}
and the last integral can be estimated as follows: by Fubini's theorem
\begin{equation*}
\begin{split}
	\int_{(1-\eta)x}^x  F(y+I) \, \dd y 
	& = 	\int_{(1-\eta)x}^x  \left( \int_\R \ind_{\{t \in (y-h,y]\}}
	\, F(\dd t) \right) \dd y \\
	& \le \int_{t \in ((1-\eta)x-h, x]}  \left( \int_\R \ind_{\{y \in [t,t+h)\}}
	\, \dd y \right) F(\dd t) \\
	& = h \, F\big((1-\eta)x-h,x]\big)
	\underset{x\to\infty}{\sim} h \, \left( \frac{1}{A((1-\eta)x)} - \frac{1}{A(x)}\right) \\
	& \underset{x\to\infty}{\sim} h \, \frac{1}{A(x)} \, 
	\left( \frac{1}{(1-\eta)^\alpha} - 1\right)
	\underset{\eta \to 0}{\sim} h \, \, \frac{1}{A(x)} \, \alpha \, \eta \,.
\end{split}
\end{equation*}
Recalling \eqref{eq:male}, it follows that \eqref{eq:suffcond1/2red} holds.
This proves the first part of Theorem~\ref{th:1/2}.

\smallskip

Next we observe that if $F$ satisfies \eqref{eq:tail1}, then necessarily
$F(x+I) = o(1/A(x))$ as $x\to\infty$. Interestingly, this bound can
be approached as close as one wishes, in the following sense.

\begin{lemma}\label{lem:uao}
Fix two arbitrary
positive sequences $(z_n)_{n\in\N}$, $(\epsilon_n)_{n\in\N}$ such that
$z_n \to \infty$ and $\epsilon_n \to 0$.
For any $A(x) \in \cR_\alpha$, with $\alpha \in (0,1)$, there are
a constant $c \in (0,\infty)$, a
subsequence $(n_k)_{k\in\N}$ of $n$ and a probability $F$ on $(0,\infty)$
satisfying \eqref{eq:tail1} such that
\begin{equation} \label{eq:uao}
	F(\{z_{n_k}\}) \ge c \, \frac{\epsilon_{n_k}}{A(z_{n_k})} \,, \qquad \forall k \in \N \,.
\end{equation}
\end{lemma}

With Lemma~\ref{lem:uao} at hand, we prove the second part of Theorem~\ref{th:1/2}.
Assume that $A(x) \in \cR_{1/2}$ is such that condition \eqref{eq:cond} fails,
that is there is a sequence $(x_n)_{n\in\N}$ with $x_n \to \infty$ such that
\begin{equation} \label{eq:dive}
	\zeta_n := \frac{L^*(x_n)}{L(x_n)} \to \infty \,.
\end{equation}
By \eqref{eq:L*}, since $L(\cdot)$ is continuous, we can write
$L^*(x_n) = L(s_n)$ for some $1 \le s_n \le x_n$.
We recall that, for any $\epsilon > 0$, one has
$L(s) / L(x_n) \to 1$ uniformly for $s \in [\epsilon x_n, x_n]$,
by the uniform convergence theorem of slowly varying functions
\cite[Theorem~1.2.1]{cf:BinGolTeu}.
Then it follows by \eqref{eq:dive} that necessarily $s_n = o(x_n)$.
Summarizing:
\begin{equation*}
	x_n \to \infty \,, \qquad s_n = o(x_n) \,, \qquad
	\zeta_n \to \infty \qquad \text{with} \qquad
	\zeta_n = \frac{L(s_n)}{L(x_n)} \,.
\end{equation*}
Let us define
\begin{equation*}
	z_n := x_n - s_n\,, \qquad
	\epsilon_n := \frac{1}{\zeta_n} \,,
\end{equation*}
so that $z_n \sim x_n \to \infty$ and $\epsilon_n \to 0$.
By Lemma~\ref{lem:uao}, there are a subsequence
$(n_k)_{k\in\N}$ of $n$ and a probability $F$ on $(0,\infty)$ such that
\eqref{eq:uao} holds. Then, by $A(x) = L(x) \sqrt{x}$,
\begin{equation*}
\begin{split}
	\int_1^{\eta x_{n_k}} \frac{A(s)^2}{s} \, F( x_{n_k} -\dd s) 
	& \ge \frac{A(s_{n_k})^2}{s_{n_k}} \, F(\{ x_{n_k} - s_{n_k}\})
	= L(s_{n_k})^2 \, F(\{z_{n_k}\}) \\
	& = \zeta_{n_k}^2 \, L(x_{n_k})^2 \, F(\{z_{n_k}\}) 
	\ge \zeta_{n_k}^2 \, \frac{A(x_{n_k})^2}{x_{n_k}} \,
	c \, \frac{\epsilon_{n_k}}{A(z_{n_k})}
	\gtrsim \zeta_{n_k} \, c \, \frac{A(x_{n_k})}{x_{n_k}} \,,
\end{split}
\end{equation*}
where in the last inequality we used the definition of $\epsilon_n$
and the fact that $A(x_n) \sim A(z_n)$, since $x_n \sim z_n$. Consequently,
condition \eqref{eq:ns12} is \emph{not} satisfied, because for every $\eta > 0$
\begin{equation*}
	\limsup_{x\to\infty} \frac{x}{A(x)} \int_1^{\eta x} \frac{A(s)^2}{s} \, F( x_{n_k} -\dd s) 
	\ge \limsup_{k\to\infty} \, c \, \zeta_{n_k} = \infty \,.
\end{equation*}
Since \eqref{eq:ns12} ---which is equivalent to \eqref{eq:ns12'bis0}---
is necessary for the SRT \eqref{eq:SRT},
we have built an example of $F$ satisfying \eqref{eq:tail1}
but not \eqref{eq:SRT}, completing the proof of Theorem~\ref{th:1/2}. \qed

\subsection{Proof of Lemma~\ref{lem:uao}}

Fix $n_0 \in \N$ such that $c_1 := \sum_{n \ge n_0 + 1} \frac{2 \, \alpha}{n \, A(n)} < 1$.
Then define a probability $F_1$ on $\N$ by
\begin{equation} \label{eq:argu}
	F_1(\{n\}) := \, (1-c_1) \ind_{\{n=n_0\}} \,+ \, \frac{2 \, \alpha}{n \, A(n)}
	\ind_{\{n \ge n_0+1\}} \,,
\end{equation}
so that 
\begin{equation*}
	F_1((x,\infty)) \sim \frac{2}{A(x)} \qquad \text{as} \quad x \to \infty \,.
\end{equation*}

We may assume that $(x_n)_{n\in\N}$ is increasing.
Fix a subsequence $(n_k)_{k\in\N}$ of $n$ such that
\begin{equation} \label{eq:geome}
	\frac{\epsilon_{n_{k+1}}}{A(x_{n_{k+1}})} \le \frac{1}{2}
	\, \frac{\epsilon_{n_{k}}}{A(x_{n_{k}})} \,, \qquad \forall k \in \N \,,
\end{equation}
which is clearly possible since $A(x_{n_{k+1}}) \ge A(x_{n_{k}})$ and
$\epsilon_n \to 0$. Then define a probability $F_2$ supported
by $E := \{x_{n_k}: \ k\in\N\}$ by
\begin{equation*}
	F_2(\{x_{n_k}\}) := c_2 \, \frac{\epsilon_{n_k}}{A(x_{n_k})} \,, \qquad
	\text{where} \qquad
	c_2 := \left( \sum_{k\in\N} \frac{\epsilon_{n_k}}{A(x_{n_k})} \right)^{-1} \,,
\end{equation*}
and note that $c_2 > 0$ because the series converges, by \eqref{eq:geome}.
Given $x \in (0,\infty)$, if we define $\bar k(x) := \min\{k\in\N: \ x_{n_k} > x\}$,
using \eqref{eq:geome} we can write
\begin{equation*}
	F_2((x,\infty)) =
	\sum_{k \ge \bar k(x)} c_2 \, \frac{\epsilon_{n_k}}{A(x_{n_k})}
	\le c_2 \, \frac{\epsilon_{n_{\bar k(x)}}}{A(x_{n_{\bar k(x)}})}
	\sum_{k \ge \bar k(x)} \frac{1}{2^{k-\bar k(x)}}
	\le c_2 \,
	\frac{\epsilon_{n_{\bar k(x)}}}{A(x)} \, 2 \,,
\end{equation*}
where the last inequality holds because $x_{n_{\bar k(x)}} \ge x$
by construction. Since $\epsilon_n \to 0$, we have shown that
$F_2((x,\infty)) = o(1/A(x))$ as $x\to\infty$.

We can finally define the probability $F := \frac{1}{2}(F_1 + F_2)$, which
satisfies \eqref{eq:tail1} since
\begin{equation*}
	F((x,\infty)) \sim \frac{1}{2} \big(F_1((x,\infty)) + F_2((x,\infty))\big)
	\sim \frac{1}{2}
	\left( \frac{2}{A(x)} + o\left(\frac{1}{A(x)}\right)
	\right) \sim \frac{1}{A(x)} \,,
\end{equation*}
and by construction $F(\{x_{n_k}\}) \ge \frac{1}{2} F_2(\{x_{n_k}\})$,
hence \eqref{eq:uao} holds with $c := c_2/2$.\qed

\subsection{Proof of Theorem~\ref{th:cont}}

The case $\alpha > \frac{1}{2}$ was already considered in Section~\ref{sec:>12},
hence we focus on $\alpha \le \frac{1}{2}$.
Since the necessity of \eqref{eq:ns12'bis0two} is proved in Appendix~\ref{sec:necce},
it remains to give examples of
$F$ satisfying \eqref{eq:tail2}
and \eqref{eq:ns12'bis0} but not \eqref{eq:ns12'bis0two}.

\smallskip

We first consider the case $\alpha < \frac{1}{2}$.
We fix $A(x) := x^\alpha$
and, in analogy with \eqref{eq:argu}, we define a \emph{symmetric} probability
$F_1$ on $\Z$ by
\begin{equation} \label{eq:f1a}
	F_1(\{n\}) :=  c_1 \, \ind_{\{|n|=n_0\}} \,+\,
	\frac{2 \alpha}{n \, A(n)} \, \ind_{\{|n| \ge n_0+1\}} \,, 
\end{equation}
where $c_1 \in (0,1)$ and $n_0 \in \N$ are chosen so that $\sum_{n\in\Z} F_1(\{n\}) = 1$.
Note that
\begin{equation} \label{eq:f1b}
	F_1((-\infty,-x]) \sim 
	F_1((x,\infty)) \sim \frac{2}{x^\alpha}
	= \frac{2}{A(x)} \qquad \text{as} \quad x \to \infty \,.
\end{equation}

For $n,k\in\N$ we define (recall that $\alpha < \frac{1}{2}$)
\begin{equation} \label{eq:zk}
	x_n := 2^n \,, \quad \
	z_k := k^{\frac{1}{1-2\alpha}},  \quad \
	E_n := \{x_{n,k} := x_n + z_k: \ 0 \le k <
	\hat k_n := \lfloor x_n^{1-2\alpha} \rfloor\} 
\end{equation}
so that $E_n$ is a finite set of points in $[x_n, x_{n+1})$. 
Since $|E_n| \le 2 x_n^{1-2\alpha}$, we have
\begin{equation} \label{eq:dn}
	\sum_{y \in E_{n}} \frac{A(y)}{y\, \sqrt{\log y}}
	\le \frac{A(x_n)}{x_n\, \sqrt{\log x_n}} \, | E_{n} |
	\le \frac{2 \, x_n^{1-2\alpha}}{x_n^{1-\alpha} \sqrt{\log x_n}}
	= \frac{2}{x_n^\alpha \, \sqrt{\log x_n}} =: d_n \,,
\end{equation}
and note that $\sum_{n\in\N} d_n < \infty$,
since $x_n = 2^n$.
We can then define a probability $F_2$ by
\begin{equation} \label{eq:F2}
	F_2(\{y\}) := c_2 \, \frac{A(y) \, \ind_{\{y\in E\}}}{y \, \sqrt{\log y}}
	= c_2 \, \frac{\ind_{\{y\in E\}}}{y^{1-\alpha} \sqrt{\log y}} \,,
	\qquad \text{where} \qquad E := \bigcup_{n\in\N} E_{n} \,,
\end{equation}
and $c_2$ is a normalizing constant. Note that
for $x \in [x_\ell, x_{\ell+1})$ we have the upper bound
\begin{equation} \label{eq:tailf2}
\begin{split}
	F_2((x, \infty)) & \le \sum_{n=\ell}^\infty 
	F_2(E_{n}) \le \sum_{n=\ell}^\infty 
	d_n  \le
	\frac{c_2}{\sqrt{\log x_\ell}}
	\sum_{n=\ell}^\infty \frac{1}{2^{\alpha n}}
	\lesssim \frac{1}{\sqrt{\log x_\ell}} \, \frac{1}{2^{\alpha\ell}} 
	= \frac{1}{x_\ell^\alpha \sqrt{\log x_\ell}} \\
	& \lesssim \frac{1}{x^\alpha \sqrt{\log x}} =
	o\left(\frac{1}{x^\alpha}\right) =
	o\left(\frac{1}{A(x)}\right)  \qquad \text{as } x\to\infty \,.
\end{split}
\end{equation}
Consequently, the probability $F := \frac{1}{2}(F_1 + F_2)$
satisfies \eqref{eq:tail2} with $A(x) = x^\alpha$ and $p = q = 1$.

Let us show that $F$ does not satisfy \eqref{eq:ns12two}, 
which is equivalent to \eqref{eq:ns12'bis0two}. We focus on the second part of the integral. 
For $\eta < \frac{1}{2}$ and $x=x_n$, so that $[x_n + 1, x_n + \eta x_n) \subseteq E_n$,
we have
\begin{equation*}
\begin{split}
	\int_{[1,\eta x_n)} \frac{A(s)^2}{s} \, F_2(x_n + \dd s) 
	& = \sum_{y \in E_{n}} 
	\frac{\ind_{\{y \in [x_n+1, x_n+\eta x_n)\}}}{(y-x_n)^{1-2\alpha}} \, F_2(\{y\}) \\
	& \ge F_2(\{x_{n+1}\})
	\sum_{1 \le k < (\eta x_n)^{1-2\alpha}} \frac{1}{z_k^{1-2\alpha}}\,,
\end{split}
\end{equation*}
because $F_2(\{\cdot\})$ is decreasing on $E$. Recalling \eqref{eq:zk}-\eqref{eq:F2},
since $\sum_{k=1}^z \frac{1}{k} \sim \log z$, we obtain
\begin{equation*}
	\int_{[1,\eta x_n)} \frac{A(s)^2}{s} \, F_2(x_n + \dd s) 
	\gtrsim_\eta F_2(\{x_{n+1}\}) \log x_n \gtrsim
	\frac{\sqrt{\log x_n}}{x_n^{1-\alpha}} = \sqrt{\log x_n} \, \frac{A(x_n)}{x_n} \,.
\end{equation*}
The $\limsup_{x\to\infty}$ in \eqref{eq:ns12two} then equals $\infty$
for every fixed $\eta > 0$, hence \eqref{eq:ns12two} does not hold.

Let us finally show that $F$ does satisfy \eqref{eq:ns12}, which is equivalent to \eqref{eq:ns12'bis0}
by Lemma~\ref{lem:tec}. Since $F_1$ clearly satisfies \eqref{eq:ns12}, it suffices to focus on $F_2$.
Note that
\begin{equation} \label{eq:yie}
\begin{split}
	\int_{[1,\eta x)} \frac{A(s)^2}{s} \, F_2(x-\dd s) 
	& \le \left\{ \sup_{z \in (x-\eta x, x-1]} F_2(\{z\}) \right\}
	\sum_{y \in E} \frac{1}{(x-y)^{1-2\alpha}} \,
	\ind_{\{y \in (x-\eta x, x-1]\}} \\
	& \lesssim \frac{1}{x^{1-\alpha} \sqrt{\log x}} \, 
	\sum_{y \in E} \frac{1}{(x-y)^{1-2\alpha}} \, \ind_{\{y \in (x-\eta x, x-1]\}} \,.
\end{split}
\end{equation}
For $x \ge 5$ we have $x-1 \in [x_\ell, x_{\ell+1})$ for some $\ell \ge 2$.
For $\eta < \frac{1}{2}$, certainly $x-\eta x > \frac{x}{2} \ge x_{\ell-1}$,
hence we can replace $\ind_{\{y \in (x-\eta x, x-1]\}}$
by $\ind_{\{y \in [x_{\ell-1}, x-1]\}}$ in \eqref{eq:yie},
getting
\begin{equation}\label{eq:2so}
	\int_{[1,\eta x)} \frac{A(s)^2}{s} \, F_2(x-\dd s) 
	\lesssim \frac{A(x)}{x\, \sqrt{\log x}} 
	\left\{ 
	\sum_{y \in E_{\ell-1}} 	\frac{1}{(x-y)^{1-2\alpha}} +
	\sum_{y \in E_{\ell}} \frac{\ind_{\{y \le x-1\}}}{(x-y)^{1-2\alpha}} 
	\right\}\,.
\end{equation}
It suffices to show that both sums
are uniformly bounded, and relation \eqref{eq:ns12} holds.

We start looking at the second sum. 
Writing $y= x_{\ell,k}$, by \eqref{eq:zk}, the constraint
$y \le x-1$ becomes $k \le \bar k$ for a suitable $\bar k = \bar k_x$
(the precise value is immaterial), hence
\begin{equation} \label{eq:2so2}
	\sum_{y \in E_{\ell}} \frac{\ind_{\{y \le x-1\}}}{(x-y)^{1-2\alpha}}
	= \sum_{0 \le k \le \bar k} \frac{1}{(x-x_{\ell,k})^{1-2\alpha}}
	\le 1 + \sum_{0 \le k \le \bar k - 1}
	\frac{1}{(x_{\ell, \bar k}-x_{\ell,k})^{1-2\alpha}} \,,
\end{equation}
where we have bounded the term $k=\bar k$ by $x - x_{\ell,\bar k} \ge x - (x-1) = 1$,
while for the terms $k < \bar k$ we have replaced $x$ by $x_{\ell,\bar k} < x$.
Next observe that for $k = \bar k - i$
\begin{equation*}
	x_{\ell, \bar k}-x_{\ell,\bar k - i} = z_{\bar k}-z_{\bar k - i} =
	\bar k^{\frac{1}{1-2\alpha}} - (\bar k - i)^{\frac{1}{1-2\alpha}}
	= \bar k^{\frac{1}{1-2\alpha}} 
	\left[1 - (1-\tfrac{i}{\rule{0pt}{.65em}\bar k})^{\frac{1}{1-2\alpha}}\right] \,.
\end{equation*}
Since $1-(1-x)^\gamma \ge x$ for $0 \le x \le 1$ and $\gamma \ge 1$, we obtain
$x_{\ell, \bar k}-x_{\ell,\bar k - i} \ge \bar k^{\frac{2\alpha}{1-2\alpha}} i$, hence
\begin{equation*}
	\sum_{y \in E_{\ell}} \frac{\ind_{\{y \le x-1\}}}{(x-y)^{1-2\alpha}}
	\le 1 + \sum_{1 \le i \le \bar k} 
	\frac{1}{\big(\bar k^{\frac{2\alpha}{1-2\alpha}} 
	i \big)^{1-2\alpha}} = 1 + \frac{1}{\bar k^{2\alpha}}
	\sum_{1 \le i \le \bar k} \frac{1}{i^{1-2\alpha}} \lesssim 1 \,,
\end{equation*}
uniformly over $\bar k$, by \eqref{eq:Kar}. Analogously,
for the first sum in \eqref{eq:2so}, we can write $y = x_{\ell-1,k}$
and sum over $0 \le k \le \hat k$
with $\hat k := \hat k_{\ell - 1}$ (recall \eqref{eq:zk}).
Arguing as before, we can bound
\begin{equation*}
	\sum_{y \in E_{\ell-1}} \frac{1}{(x-y)^{1-2\alpha}}
	= \sum_{0 \le k \le \hat k} \frac{1}{(x-x_{\ell-1,k})^{1-2\alpha}}
	\le 1 + \sum_{0 \le k \le \hat k - 1}
	\frac{1}{(x_{\ell-1, \hat k}-x_{\ell-1,k})^{1-2\alpha}} \,,
\end{equation*}
and also this sum is $\lesssim 1$, by the previous steps with $\hat k$ in place of $\bar k$. 

\medskip

We finally consider the case $\alpha = \frac{1}{2}$.
We fix $A(x) := \sqrt{x}/\log (1+x)$
and we define $F_1$ as in \eqref{eq:f1a} (with our current $A(x)$), so that \eqref{eq:f1b} holds.
Next we change \eqref{eq:zk} to
\begin{equation*} 
	x_n := 2^n \,, \quad \
	z_k := e^{\sqrt{k}}-1,  \quad \
	E_n := \{x_{n,k} := x_n + z_k: \ 0 \le k <
	\hat k_n := \lfloor \log(1+ x_n) \rfloor^{2}\} \,,
\end{equation*}
and note that $E_n \subseteq [x_n, x_{n+1})$. We then define 
a probability $F_2$ supported by $E := \bigcup_{n\in\N} E_n$:
\begin{equation*}
	F_2(\{z\}) := c_2 \, \frac{A(y) \, \ind_{\{y\in E\}}}{ y \, \sqrt{\log\log (1+y)}}
	= c_2 \, \frac{\ind_{\{y\in E\}}}{\sqrt{y} \, \log (1+y) \, \sqrt{\log\log (1+y)}} \,.
\end{equation*}
Since $|E_n| \le 2 (\log (1+x_n))^{2}$, we can write
\begin{equation*}
	\sum_{y \in E_{n}} \frac{A(y)}{y \, \sqrt{\log \log (1+y)}}
	\le \frac{ | E_{n} |}{\sqrt{x_n} \, \log(1+ x_n) \, \sqrt{\log \log(1+ x_n)}}
	\le \frac{2 \, \log(1+ x_n)}{\sqrt{x_n} \, \sqrt{\log\log(1+ x_n)}} =: d_n \,,
\end{equation*}
hence for $x \in [x_\ell, x_{\ell+1})$ we have the upper bound
\begin{equation*}
	F_2((x,\infty)) \le \sum_{n=\ell}^\infty F_2(E_n)
	\le c_2 \sum_{n=\ell}^\infty d_n
	\lesssim d_\ell
	\lesssim \frac{\log(1+ x)}{\sqrt{x} \, \sqrt{\log\log(1+ x)}}
	= o\left(\frac{1}{A(x)}\right) \,.
\end{equation*}
It follows that $F := \frac{1}{2}(F_1 + F_2)$ 
satisfies \eqref{eq:tail2} with $A(x) = \sqrt{x}/\log (1+x)$ and $p = q = 1$.

To show that $F$ does not satisfy \eqref{eq:ns12two}, 
note that for $\eta < \frac{1}{2}$ and $x=x_n$ we have
\begin{equation*}
\begin{split}
	\int_{[1,\eta x_n)} \frac{A(s)^2}{s} \, & F(x_n + \dd s) 
	\ge F_2(\{x_{n+1}\})
	\sum_{1 \le k \le \lfloor \log (1+\eta x_n)\rfloor^{2}} \frac{1}{(\log (1+z_k))^{2}} \\
	& \gtrsim
	\frac{A(x_n)}{ x_n \, \sqrt{\log\log (1+x_n)}} \, \log\{
	\lfloor \log(1+\eta x_n)\rfloor^2\}
	\gtrsim_\eta  \frac{A(x_n)}{ x_n}
	\sqrt{\log \log x_n} \,.
\end{split}
\end{equation*}
Finally, to show that $F$ satisfies \eqref{eq:ns12}, arguing as in \eqref{eq:yie}
we get the analogue of \eqref{eq:2so}:
\begin{equation}\label{eq:sum12}
\begin{split}
	\int_{[1,\eta x)} \frac{A(s)^2}{s} \, F_2(x-\dd s) 
	\lesssim \frac{A(x)}{x \sqrt{ \log \log (1+x)}} 
	\Bigg\{ &
	\sum_{y \in E_{\ell-1}} 	\frac{1}{[\log (1+x-y)]^{2}} \\
	& \qquad +
	\sum_{y \in E_{\ell}} 
	\frac{\ind_{\{y \le x-1\}}}{[\log (1+x-y)]^{2}} 
	\Bigg\} \,,
\end{split}
\end{equation}
and it remains to show that both sums are bounded. For a suitable
$\bar k = \bar k_x$ the second sum is
\begin{equation} \label{eq:2so2bis}
	\sum_{0 \le k \le \bar k} 
	\frac{1}{[\log(1+x-x_{\ell,k})]^{2}}
	\le \frac{1}{(\log 2)^2} + \sum_{0 \le k \le \bar k - 1}
	\frac{1}{[\log(1+x_{\ell, \bar k}-x_{\ell,k})]^{2}} \,,
\end{equation}
where we have bounded the term $k=\bar k$ by $x-x_{\ell, \bar k} \ge x-(x-1) = 1$
and we have replaced $x$ by $x_{\ell,\bar k}$ in the remaining terms.
Next we note that for all $k \le \bar k - 1$
\begin{equation*}
	\log(1+x_{\ell, \bar k}-x_{\ell,k}) \ge 
	\log(1+x_{\ell, \bar k}-x_{\ell,\bar k - 1})
	= \log\big( 1 + e^{\sqrt{\bar k}}-e^{\sqrt{\bar k - 1}} \big)
	\gtrsim \log \frac{e^{\sqrt{\bar k}}}{\sqrt{\bar k}} \gtrsim \sqrt{\bar k} \,,
\end{equation*}
which plugged into \eqref{eq:2so2bis} shows that the sum is uniformly bounded.
The first sum in \eqref{eq:sum12} is estimated similarly, replacing $\ell$ by $\ell - 1$
and $\bar k$ by $\hat k_{\ell-1}$. This completes the proof.
\qed

\appendix

\section{Miscellanea}
\label{sec:refoapp}

\subsection{Proof of Lemma~\ref{lem:tec}}\label{sec:tec}

By \eqref{eq:rr},
uniformly for $0 \le s \le \eta x$
and $\eta < \frac{1}{2}$, we can write
\begin{equation} \label{eq:inai}
	F(x-s+I) \sim \frac{\rr(x-s)}{(x-s) \, A(x-s)} \simeq
	\frac{\rr(x-s)}{x \, A(x)} \,,
\end{equation}
and analogously with $s$ replaced by $-s$.
Then \eqref{eq:ns12'bis0two} is equivalent to the following
relation:
\begin{equation}\label{eq:ns12'}
	\lim_{\eta \to 0}
	\left( \limsup_{x\to\infty} \frac{x}{A(x)}
	\int_{1}^{\eta x} 
	\frac{A(s)^2}{s} \, \big( F(x - s + I) +
	\ind_{\{q>0\}} F(x + s + I) \big) \, \dd s \right) = 0 \,.
\end{equation}
We show below that \eqref{eq:ns12'} is equivalent to \eqref{eq:ns12two}.
Then \eqref{eq:ns12'bis0two} is equivalent to \eqref{eq:ns12two},
i.e.\ the last statement in Lemma~\ref{lem:tec} holds.
For $q=0$, we have the equivalence of \eqref{eq:ns12'bis0} and \eqref{eq:ns12}.

Let us now prove the equivalence of relations \eqref{eq:ns} and \eqref{eq:nsbis0}.
Since $h > 0$ is fixed, uniformly for $0 \le s \le \eta x$ and $\eta < \frac{1}{2}$ we can write
\begin{equation*}
	\P(X \in (x-s,x]) \sim \P(X \in (x-s,x-h])
	= \int_\R \ind_{\{t \in [h,s)\}} \, F(x-\dd t)
\end{equation*}
Writing $1 = \frac{1}{h} \int_{\R}
\ind_{\{u \in (t-h,t]\}} \, \dd u$, for any fixed $t$, by Fubini's theorem we get
\begin{equation*}
\begin{split}
	\P(X \in (x-s,x]) & \sim \frac{1}{h} \int_{0}^s\left(
	\int_\R \ind_{\{t \in [u,u+h)\}} \, F(x-\dd t) \right) \, \dd u 
	= \frac{1}{h} \int_{0}^s F(x-u + I) \, \dd u \,.
\end{split}
\end{equation*}
Applying \eqref{eq:inai} then gives
\begin{equation*}
	\P(X \in (x-s,x]) \sim \frac{1}{h}
	\, \frac{1}{x \, A(x)} \int_{0}^s \rr(x-u + I) \, \dd u
	= \frac{1}{h}
	\, \frac{1}{x \, A(x)} \, \RR_0(x-s,x) \,,
\end{equation*}
which shows that \eqref{eq:ns} is equivalent to \eqref{eq:nsbis0}.

It remains to prove the equivalence of \eqref{eq:ns12'} and \eqref{eq:ns12two}.
We recall that $I=(-h,0]$ and, for this purpose, we can take $h > 0$ arbitrarily
also in the lattice case.
We first claim that in \eqref{eq:ns12two} 
one can equivalently replace the domain of integration
$[1,\eta x)$ by $[1+h, \eta x)$.
For this it is enough to show that the interval $[1,1+h)$ gives a contribution
to \eqref{eq:ns12two} which is dominated by that of $[1+h,1+2h)$.
The function $A(s)^2/s$ is continuous and strictly positive, hence it is
bounded away from zero and infinity
in any compact interval. Then for $x, x'$ large enough
\begin{gather*}
	\frac{x}{A(x)}
	\int_{s \in [1,1+h)} 
	\frac{A(s)^2}{s} \, \P(X \in x - \dd s) 
	\lesssim \frac{x}{A(x)} \, \P(X \in (x-h-1, x-1]) \,, \\
	\frac{x'}{A(x')}
	\int_{s \in [1+h, 1+2h)} 
	\frac{A(s)^2}{s} \, \P(X \in x' - \dd s) 
	\gtrsim \frac{x'}{A(x')} \, \P(X \in (x'-2h-1, x'-h-1]) \,.
\end{gather*}
Choosing $x' = x+h$ and letting $x \to \infty$, since 
$\frac{x'}{A(x')} \sim \frac{x}{A(x)}$, we have proved the claim.
With analogous estimates one deals with
$\P(X \in x+\dd s)$ in \eqref{eq:ns12two}.

Next we note that there are constants $0 < c < C < \infty$
(depending on $w$) such that
\begin{equation}\label{eq:cC}
	c \left( \frac{1}{h} \int_{s-h}^{s} \frac{A(t)^2}{t} \, \dd t \right) \le
	\frac{A(s)^2}{s}
	\le C \left( \frac{1}{h} \int_{s-h}^{s} \frac{A(t)^2}{t} \, \dd t \right) \,,
	\qquad \forall s \ge 1+h \,.
\end{equation}
Plugging this into \eqref{eq:ns12two}, where the domain of integration has
been changed to $[1+h,\eta x)$,
shows precisely that \eqref{eq:ns12two} is equivalent to \eqref{eq:ns12'}.\qed

\subsection{Proof of Theorem~\ref{th:main}: second part}
\label{sec:maramao}

We show that condition \eqref{eq:ns} is equivalent to
\eqref{eq:ns12} for $\alpha < \frac{1}{2}$, while it is stronger
for $\alpha = \frac{1}{2}$. By Lemma~\ref{lem:tec}, an analogous
statement holds for \eqref{eq:nsbis0} and \eqref{eq:ns12'bis0},
proving the second part of Theorem~\ref{th:main}.

For fixed $x$, we define 
$G(s) := \P(x-X \in (1,s]) = \P(x-X \le s) - \P(x-X \le 1)$ and note that
$\P(x-X \in \dd s) = \dd G(s)$. 
Integrating by parts, since $G(1) = 0$ we get
\begin{equation} \label{eq:parts}
\begin{split}
	\int_{s \in [1,\eta x)} 
	\frac{A(s)^2}{s} \, \P(x-X \in \dd s)
	& = G(\eta x-) \, \frac{A(\eta x)^2}{\eta x} -
	\int_1^{\eta x} G(s) \, \frac{\dd}{\dd s} \left( \frac{A(s)^2}{s} \right)
	\dd s \,.
\end{split}
\end{equation}
The first term in the right hand side equals 
\begin{equation} \label{eq:flo}
\begin{split}
	\P(X \in (x-\eta x, x-1)) \, \frac{A(\eta x)^2}{\eta x}
	& \underset{x\to\infty}{\sim} 
	\left(\frac{1}{A((1-\eta)x)}-\frac{1}{A(x-1)}\right) \, \frac{A(\eta x)^2}{\eta x} \\
	& \underset{x\to\infty}{\sim} 
	\left(\frac{1}{(1-\eta)^\alpha}- 1\right) \eta^{2\alpha-1} \, \frac{A(x)}{x}
	=
	O( \eta^{2\alpha}) \, \frac{A(x)}{x} \,,
\end{split}
\end{equation}
hence this terms always gives a negligible contribution
to the limit in \eqref{eq:ns12}.

Next observe that by \eqref{eq:deriv}
\begin{equation*}
	\frac{\dd}{\dd s} \frac{A(s)^2}{s} 
	= \frac{2A(s)A'(s)}{s} - \frac{A(s)^2}{s^2}
	\begin{cases}
	\underset{s\to\infty}{\sim} 	
	(2\alpha - 1) \frac{A(s)^2}{s^2} & \text{if } \ \alpha < \frac{1}{2} \\
	\rule{0pt}{1.5em}\underset{s\to\infty}{=} 
	o\left( \frac{A(s)^2}{s^2} \right) & \text{if } \ \alpha = \frac{1}{2} 
	\end{cases} \,.
\end{equation*}
If $\alpha < \frac{1}{2}$, for the second term in \eqref{eq:parts} we can write
\begin{equation} \label{eq:relaa}
\begin{split}
	-\int_1^{\eta x} G(s) \, \frac{\dd}{\dd s} \left( \frac{A(s)^2}{s} \right)
	& \simeq \int_1^{\eta x} \frac{A(s)^2}{s^2} \, \P(X \in [x-s,x-1)) \, \dd s \\
	& = \int_1^{\eta x} \frac{A(s)^2}{s^2} \, \P(X \in (x-s,x-1)) \, \dd s \,.
\end{split}
\end{equation}
If relation \eqref{eq:ns} holds, it follows by
\eqref{eq:parts}-\eqref{eq:flo}-\eqref{eq:relaa} that relation \eqref{eq:ns12}
also holds. Viceversa, if \eqref{eq:ns12} holds,
applying again \eqref{eq:parts}-\eqref{eq:flo}-\eqref{eq:relaa} together with
\eqref{eq:nec} (which is a consequence of \eqref{eq:ns12}), we see that
\eqref{eq:ns} holds. Thus \eqref{eq:ns} and \eqref{eq:ns12}
are equivalent for $\alpha < \frac{1}{2}$.

For $\alpha = \frac{1}{2}$ we can replace $\simeq$ by $\lesssim$ in 
\eqref{eq:relaa}, hence \eqref{eq:ns} still implies \eqref{eq:ns12}.\qed

\subsection{Necessity of \eqref{eq:ns12'bis0} and \eqref{eq:ns12'bis0two} for the SRT}
\label{sec:necce}

We assume that $F$ is a probability on $\R$ satisfying \eqref{eq:tail2}.
We show that relation \eqref{eq:SRTeq}, which is equivalent to the SRT \eqref{eq:SRT},
implies \eqref{eq:ns12two}, hence it implies \eqref{eq:ns12'bis0two}, by
Lemma~\ref{lem:tec}. In particular, the case $q=0$ shows
that, assuming \eqref{eq:tail1}, relation \eqref{eq:SRTeq} implies \eqref{eq:ns12'bis0}.

\smallskip

Recall that $J = (-v,0]$, cf.\ \eqref{eq:J}.
Assume that $F$ satisfies \eqref{eq:tail2}
and define $K \subseteq \R$ by
\begin{equation}\label{eq:K}
	K := \begin{cases}
	[1,2] & \text{if } q = 0\\
	\rule{0pt}{1.3em} [-2,-1] \cup [1,2] & \text{if } q > 0
	\end{cases} \,,
\end{equation}
Here is a mild refinement of the local limit theorem \eqref{eq:llt},
there are $c, C \in (0,\infty)$ such that
\begin{equation}\label{eq:Chillt}
	\inf_{z \in \R: \, z/a_n \in K}
	\P\left(S_n \in z+J, \, \max_{1 \le i \le n} X_i \le C a_n \right)
	\ge \frac{c}{a_n} \,, \qquad \forall n \in \N.
\end{equation}
This follows by \cite[Lemma 4.5]{cf:Chi},
but it is worth giving a direct proof. By \eqref{eq:llt}, there is $c_1 > 0$ such that
\begin{equation} \label{eq:infphi}
	\inf_{z \in \R: \, z/a_n \in K} \P(S_n \in z+J) \ge \frac{c_1}{a_n} \,,
	\qquad \forall n \in \N,
\end{equation}
because $\min_{z \in K}\phi(z) > 0$. Next, for the maximum restricted
to $i \le n/2$ (assuming that $n$ is even for simplicity, the odd case is analogous), we can write
\begin{equation*}
\begin{split}
	\P\left(S_n \in z+J, \, \max_{1 \le i \le \frac{n}{2}} X_i > C a_n \right)
	& = \int_{\R} 
	\P\bigg(S_{\frac{n}{2}} \in \dd y, \, \max_{1 \le i \le \frac{n}{2}} X_i > C a_n \bigg)
	\P\left(S_{\frac{n}{2}} \in z-y+J \right) \\
	& \le \P\left( \max_{1 \le i \le \frac{n}{2}} X_i > C a_n \right)
	\left\{ \sup_{x \in \R} \P\left(S_{\frac{n}{2}} \in x+J\right) \right\} \,.
\end{split}
\end{equation*}
The term in bracket is $\le c_2 / a_n$, by \eqref{eq:sup}.
By Potter's bounds \eqref{eq:Potter} and by $A(a_n) = n$
we have $\P(X > C a_n) \le c_3/A(Ca_n) \le c_4 / (C^{\alpha/2} n)$,
hence
\begin{equation*}
	\sup_{z\in\R} \P\left(S_n \in z+J, \, \max_{1 \le i \le \frac{n}{2}} X_i > C a_n \right)
	\le \frac{n}{2} \, \frac{c_4}{C^{\alpha/2} \, n} \, \frac{c_2}{a_n}
	= \frac{c_2 \, c_4}{2 \, C^{\alpha/2} \, a_n} \,.
\end{equation*}
The contribution of $\{\max_{\frac{n}{2} \le i \le n} X_i > C a_n\}$
is the same, by exchangeability, hence by \eqref{eq:infphi}
\begin{equation*}
	\inf_{z \in \R: \, z/a_n \in K}
	\P\left(S_n \in z+J, \, \max_{1 \le i \le n} X_i \le C a_n \right)
	\ge \frac{c_1}{a_n} - 2 \frac{c_2 \, c_4}{2 \, C^{\alpha/2} a_n}
	= \left(c_1 - \frac{c_2 \, c_4}{C^{\alpha/2}}\right) \frac{1}{a_n} \,,
\end{equation*}
which proves \eqref{eq:Chillt}, provided $C$ is chosen large enough.

\smallskip

Next we argue as in \cite[Proposition 2.2]{cf:Chi}.
Since $\{X_i > t, \, \max_{j \in \{1,\ldots,n\} \setminus \{i\}} X_j \le t\}$
are disjoint events for $i=1, \ldots, n$, we can write
\begin{equation} \label{eq:start}
	\P(S_n \in x+J) \ge n \, \P\left(S_n \in x+J, \ X_n > \frac{x}{2}, \
	\max_{1 \le j \le n-1} X_j \le \frac{x}{2} \right) \,.
\end{equation}
If $n \le A(\delta x)$ then $a_n \le \delta x$, hence $\frac{x}{2} > C a_n$
for $\delta < \frac{1}{2C}$. 
Therefore, by \eqref{eq:K}-\eqref{eq:Chillt},
\begin{equation} \label{eq:lastint}
\begin{split}
	\P(S_n \in x+J) & \ge \int_{\frac{x}{2}}^\infty \P(X \in \dd y)
	 \, n \, \P\left(S_{n-1} \in x-y+J, \
	\max_{1 \le j \le n-1} X_j \le C a_n \right) \\
	& \ge \int_{\frac{x}{2}}^\infty \P(X \in \dd y)
	\, c \, \frac{n-1}{a_{n-1}} \, \ind_{\{(x-y)/a_{n-1} \in K\}} \\
	& \ge \int_{\frac{x}{2}}^\infty \P(X \in \dd y)
	\, c \, \frac{A\big(\frac{|x-y|}{2}\big)}{|x-y|} \, \ind_{\{(x-y)/a_{n-1} \in K\}} \,,
\end{split}
\end{equation}
where the last inequality follows because $a_{n-1} \le |x-y| \le 2a_{n-1}$,
by the definition \eqref{eq:K} of $K$,
and we recall that $a_m = A^{-1}(m)$.
Let us assume that $q > 0$.
If we restrict the integral to $y \in (x-\delta x, x-1] \cup [x+1, x+\delta x)$, i.e.\
$1 \le |y-x| < \delta x$, summing over $n \le A(\delta x)$ we get
\begin{equation*}
	\sum_{1 \le n \le A(\delta x)}\ind_{\{(x-y)/a_{n-1} \in K\}} =
	\sum_{1 \le n \le A(\delta x)}
	\ind_{\{A(\frac{|x-y|}{2}) \le n-1 \le A(|x-y|)\}}
	\ge A(|x-y|) - A(\tfrac{|x-y|}{2}) \,.
\end{equation*}
Since $A(z) - A(\frac{z}{2}) \gtrsim A(z)$, and also
$A(\frac{z}{2}) \gtrsim A(z)$, we obtain from \eqref{eq:lastint}
\begin{equation} \label{eq:shode}
\begin{split}
	\sum_{1 \le n \le A(\delta x)} \P(S_n & \in x+J)
	\gtrsim \int_{1 \le |y-x| < \delta x} \P(X \in \dd y) \, \frac{A(|x-y|)^2}{|x-y|} \\
	& = \int_{s \in [1,\delta x)} \frac{A(s)^2}{s} \, \big( \P(X \in x - \dd s)
	+ \ind_{\{q>0\}} \P(X \in x + \dd s) \big) \,,
\end{split}
\end{equation}
where we performed the change of variables $s = x-y$
and we inserted $\ind_{\{q>0\}}$ so that the formula holds also for $q=0$
(just restrict \eqref{eq:lastint} to $y \in (x-\delta x, x-1]$).

Assume now that \eqref{eq:SRTeq} holds.
If we can replace $I=(-h,0]$ by $J = (-v,0]$ therein,
\eqref{eq:shode} shows that \eqref{eq:ns12two} holds, completing the proof.
To replace $I$ by $J$, it suffices to write
\begin{equation} \label{eq:follo}
	\P(S_n \in x+J) \le \sum_{\ell = 0}^{\lfloor v/h \rfloor}
	\P(S_n \in x_\ell + I) \,, \qquad \text{where} \qquad
	x_\ell := x-\ell h \,,
\end{equation}
and note that relation \eqref{eq:SRTeq} holds replacing
$\P(S_n \in x + I)$ by $\P(S_n \in x_\ell + I)$, for fixed $\ell$, because
$x/A(x) \sim x_\ell/A(x_\ell)$.
(Since $v > 0$ and $h > 0$ are fixed, $\lfloor v/h \rfloor$ is also fixed.)
\qed

\subsection{On condition \eqref{eq:ns12'bis0} for $\alpha > \frac{1}{2}$}
\label{sec:equi<}

Let us show that condition \eqref{eq:ns12'bis0}
is always satisfied for $\alpha > \frac{1}{2}$. 
By Lemma~\ref{lem:tec}, it is equivalent to prove \eqref{eq:ns12}. Plainly,
\begin{equation} \label{eq:rh}
\begin{split}
	\int_{s \in [1,\eta x)} \frac{A(s)^2}{s} \, \P(X \in x - \dd s) 
	& \le \left( \sup_{s \in [1,\eta x)} \frac{A(s)^2}{s} \right) \P(X \in (x-\eta x, x]) \\
	& \lesssim \frac{A(\eta x)^2}{\eta x} \left( \frac{1}{A((1-\eta)x)} - \frac{1}{A(x)}\right) \\
	& \underset{x\to\infty}{\sim}  \frac{A(x)}{x}
	\eta^{2\alpha-1} \left(\frac{1}{(1-\eta)^\alpha}- 1\right) 
	=
	\frac{A(x)}{x} \, O(\eta^{2\alpha}) \,,
\end{split}
\end{equation}
where the second inequality holds
because $A(s)^2/s$ is regularly varying with index $2\alpha-1 > 0$
and we can apply \cite[Theorem~1.5.3]{cf:BinGolTeu}. Consequently
relation \eqref{eq:ns12} holds.\qed

\subsection{Necessity of conditions \eqref{eq:nec} and \eqref{eq:nec2}}
\label{sec:equinec}

We prove that \eqref{eq:ns12} implies \eqref{eq:nec} and \eqref{eq:nec2}.
Let us consider relation \eqref{eq:ns12} with $x$ replaced by $x+1$:
restricting the integral to $y \in [1,1+w)$, since $A(s)^2/s$ is bounded away from zero,
we get
\begin{equation*}
\begin{split}
	0 & = \lim_{\eta \to 0} \left( \limsup_{x \to \infty} \frac{x+1}{A(x+1)}
	\int_{s \in [1,1+w)} \frac{A(s)^2}{s} \, \P(X \in x+1 - \dd s) \right) \\
	& \gtrsim \limsup_{x \to \infty} \frac{x+1}{A(x+1)}\,
	\P(X \in (x-w,x]) = \limsup_{x \to \infty} \frac{x}{A(x)}\,
	\P(X \in (x-w,x])  \,,
\end{split}
\end{equation*}
which is precisely \eqref{eq:nec}.
In order to obtain \eqref{eq:nec2}, let us write
\begin{equation*}
\begin{split}
	\P(S_m & (x-w,x]) \le m \, \P\left(S_m \in (x-w,x], 
	\max_{1 \le i \le m-1} X_i \le X_m \right) \\
	& \le m \int_{y \in [0, \frac{m-1}{m}x]} 
	\, \P\left(S_{m-1} \in \dd y, \, \max_{1 \le i \le m-1} X_i \le x-y \right)
	\, \P(X \in (x-y-w,x-y]) \\
	& \le m \, \sup_{z \in [\frac{1}{m}x, x]} \P(X \in (z-w,z])
	= \sup_{z \in [\frac{1}{m}x, x]} o\left(\frac{A(z)}{z}\right)
	\lesssim o\left(\frac{A(x)}{x}\right) \,,
\end{split}
\end{equation*}
completing the proof.
\qed


\end{document}